\documentclass[11pt]{amsart}

\usepackage[margin=4cm]{geometry}

\usepackage{amsmath,amsthm,amssymb,amsfonts}
\usepackage{graphicx,color}
\usepackage{dsfont}
\usepackage{url}

\numberwithin{equation}{section}
\numberwithin{table}{section}
\numberwithin{figure}{section}

\newcommand{\B}{\mathbb{B}}
\newcommand{\E}{\mathbb{E}}

\newcommand{\R}{\mathbb{R}}

\newcommand{\1}{\mathds{1}}

\newcommand{\cum}{\mathrm{cum}}
\newcommand{\dd}{\mathrm{d}}

\newcommand{\sgn}{\mathrm{sgn}}
\newcommand{\Sym}{\mathrm{Sym}}

\newcommand{\tr}{\mathrm{tr}}
\newcommand{\Tube}{\mathrm{Tube}}

\newcommand{\Vol}{\mathrm{Vol}}

\newcommand{\cyc}{(*)}
\newcommand{\kk}{\widetilde\kappa}
\newcommand{\s}{\widetilde s}
\newcommand{\iii}{{\sqrt{-1}}}
\newcommand{\iiii}{{\textstyle\frac{1}{\sqrt{-1}}}}
\newcommand{\rhoY}{\rho_Y}

\newcommand\ko{\kappa_{0}}
\newcommand\ki{\kappa_{1}}
\newcommand\kii{\kappa_{11}}
\newcommand\kko{\widetilde\kappa_{0}}
\newcommand\kki{\widetilde\kappa_{1}}
\newcommand\kkiia{\widetilde\kappa_{11}^{a}}
\newcommand\kkiiaa{\widetilde\kappa_{11}^{aa}}
\newcommand\kkiiid{\widetilde\kappa_{111}^{d}}
\newcommand\kkiiia{\widetilde\kappa_{111}^{a}}

\newtheorem{assumption}{Assumption}[section]

\newtheorem{corollary}{Corollary}[section]

\newtheorem{lemma}{Lemma}[section]
\newtheorem{proposition}{Proposition}[section]

\newtheorem{theorem}{Theorem}[section]

\begin{document}

\title[]{Asymptotic expansion of the expected Minkowski functional for isotropic central limit random fields}

\author[]{Satoshi Kuriki}
\email{kuriki@ism.ac.jp}
\address{The Institute of Statistical Mathematics, 10-3 Midoricho, Tachikawa, Tokyo 190-8562, Japan}

\author[]{Takahiko Matsubara}
\email{tmats@post.kek.jp}
\address{Institute of Particle and Nuclear Studies, High Energy Accelerator Research Organization (KEK), Oho 1-1, Tsukuba, Ibaraki 305-0801, Japan}

\begin{abstract}
The Minkowski functionals, including the Euler characteristic statistics, are standard tools for morphological analysis in cosmology.
Motivated by cosmic research, we examine the Minkowski functional of the excursion set for an isotropic central limit random field, the $k$-point correlation functions ($k$th order cumulants) of which have the same structure as that assumed in cosmic research.
Using 3- and 4-point correlation functions,
we derive the asymptotic expansions of the Euler characteristic density,
 which is the building block of the Minkowski functional.
The resulting formula reveals the types of non-Gaussianity that cannot be captured by the Minkowski functionals.
As an example, we consider an isotropic chi-square random field and confirm that the asymptotic expansion accurately approximates the true Euler characteristic density.
\end{abstract}

\maketitle

\section{Introduction}
\subsection{Minkowski functional in cosmology}

The Minkowski functional (MF) is a fundamental concept in integral and stochastic geometry.
It is a series of geometric quantities defined for a bounded set in the Euclidean space.
In the 2-dimensional case, the Minkowski functional of the set $M$ is a triplet consisting of the area $\Vol_2(M)$, half-length of the boundary $\frac{1}{2}\Vol_1(\partial M)$, and Euler characteristic (EC) $\chi(M)$ times $\pi$.
The Minkowski functional measures the morphological features of $M$ in a different way from conventional moment-type statistics and has been used in various scientific fields.

In cosmology, the Minkowski functional was introduced around the 1990s, 
and was first used to analyze the large-scale structure of the universe, and then the cosmic microwave background (CMB).
In particular, the Minkowski functional for the excursion set of the smoothed CMB map was first analyzed by \cite{schmalzing-buchert:1997} (cf.\,\cite{schmalzing-gorski:1998}).
The CMB radiation provides rich information on the early stages of the universe.
Its signal is recognized as an isotropic and nearly Gaussian random field.
However, hundreds of inflationary models are available to infer the various types of non-Gaussianity.
The Minkowski functional is used for the selection of such candidate models.

More precisely, let $X(t)$, $t\in T\subset\R^n$, be such a random field.
The sup-level set with a threshold $x$,
\[
 T_x = \{t\in T \,|\, X(t)\ge x\} = X^{-1}([x,\infty)),
\]
is referred to as the \textit{excursion set}.
Subsequently, the Minkowski functional curves $\mathcal{M}_j(T_x)$, $0\le j\le n$, can be calculated as a function of $x$.
The departure of the sample Minkowski functional from the expected Minkowski functional under the assumed model is used as a measure for the selection of models.

When $X(t)$, $t\in M$, is Gaussian, the Minkowski functional density (i.e., the expected MF per unit volume) is explicitly known (see (\ref{Xi_infty})).
However, the expected Minkowski functional for a non-Gaussian random field is unknown (except for the Gaussian related random fields \cite[Section 5.2]{adler-taylor:2011}).
In cosmology, a weak non-Gaussianity is expressed through the $k$-point correlation function, or the $k$th order cumulant, for $k\ge 2$ as
\begin{equation}
\label{cum_order}
 \cum(X(t_1),\ldots,X(t_k)) = O\bigl(\nu^{k-2}\bigr), \quad \nu\ll 1,
\end{equation}
where $\nu$ is a non-Gaussianity parameter
 (\cite{matsubara:2003}).
These asymptotics are the same as that of the central limit random field introduced by \cite{chamandy-worsley-taylor-gosselin:2008}.
That is, for independent and identically distributed random fields $Z_{(i)}(t)$, $t\in T$ ($i\ge 1$)
with zero mean and unit variance, the \textit{central limit random field} defined by
\begin{equation}
\label{cl_field}
 X(t) = X_N(t) = \frac{1}{\sqrt{N}}\sum_{i=1}^N Z_{(i)}(t), \ \ t\in T\subset\R^n,
\end{equation}
has a cumulant of the form (\ref{cum_order}) with $\nu=1/\sqrt{N}$.
This is a typical weakly non-Gaussian random field when $N$ is large.
In this study, we discuss the asymptotic expansion of the Minkowski functionals in the framework of the central limit random field.

Cosmology and astrophysics research is often based on massive numerical simulations.
Because the computational cost of the simulator is extremely high,
analytic methods, including the asymptotic expansion,
can be immensely helpful in research if they provide the same information as simulators.

Finally, the use of Minkowski functions in astrophysics and recent related topics are further discussed.
The systematic application of the Minkowski functional to Planck CMB data was reported in \cite{planck:2014}.
\cite{fantaye-marinucci:2015} discuss the use of the Minkowski functional for the CMB data in the presence of sky masks.
In addition to the CMB, the Minkowski functionals have been applied to various types of survey observations: 2-dimensional maps of the weak-lensing field of galaxy surveys, and 3-dimensional maps of the large-scale structure of the universe probed by galaxy distributions.
The exact treatment of the 2-dimensional cosmological data as a random field on the celestial sphere is studied (e.g., \cite{marinucci-peccati:2011}, \cite{fantaye-marinucci:2015}).
For recent developments in the applications of the Euler characteristic, Minkowski functional, and related geometric methods,
see \cite{pranav-etal:2019:betti}.

\subsection{Scope of this paper}

In reality, the index set (i.e., the survey area) $T$ is a bounded domain, and boundary corrections should always be incorporated.
Let $\mathcal{M}_{j}(T)$, $0\le j\le n$, be the Minkowski functional of the domain $T$.
Then, if $T$ is a $C^2$-stratified manifold of positive reach (see Section \ref{sec:tube} for the definition), it is known that 
\begin{align}
\label{euler}
 \E[\chi(T_x)]
 = \sum_{j=0}^{n} \mathcal{L}_{j}(T)\,\Xi_{j,N}(x),
\end{align}
if the expectation exists, where
$\mathcal{L}_{j}(T)=\omega_{n-j}^{-1}\binom{n}{j}\mathcal{M}_{n-j}(T)$ is referred to as the Lipschitz-Killing curvature of $T$, $\omega_{j}$ is the volume of the unit ball in $\R^j$, and $\Xi_{j,N}(x)$ is the Euler characteristic density of the random field $X_N(t)$ in (\ref{cl_field}) restricted to the $j$-dimensional linear subspace.
See \cite{worsley:1995}, or \cite[Theorem 1]{chamandy-worsley-taylor-gosselin:2008} for a proof based on Hadwiger's theorem.

Moreover, the expected Minkowski functionals are expressed as
\begin{align}
\label{minkowski}
 \E[\mathcal{L}_{k}(T_x)]
 = \omega_{n-k}^{-1}\binom{n}{k}\E[\mathcal{M}_{n-k}(T_x)]
 = \sum_{j=0}^{n-k} \genfrac{[}{]}{0pt}{}{k+j}{k} \mathcal{L}_{k+j}(T)\,\Xi_{j,N}(x),
\end{align}
where $\genfrac{[}{]}{0pt}{}{k+j}{k}
 = \bigl[\Gamma(\frac{k+j+1}{2})\Gamma(\frac{1}{2})\bigr]/\bigl[\Gamma(\frac{k+1}{2})\Gamma(\frac{j+1}{2})\bigr]$.
A proof based on Crofton's theorem is presented in the next section (Section \ref{sec:tube}).

In this study, we obtain the asymptotic expansion formula for the Euler characteristic density $\Xi_{n,N}(x)$ for a large $N$ and examine the effect of the non-Gaussianity.
The resulting formula also automatically provides the asymptotic expansion formula for the expected Euler characteristic and the expected Minkowski functionals via (\ref{euler}) and (\ref{minkowski}), respectively.
Prior research \cite{matsubara:2003,hikage-komatsu-matsubara:2006} provided a perturbation formula for the expected Minkowski functional up to $O(\nu)$ 
 ($\nu=N^{-\frac{1}{2}}$) 
using 3-point correlation function for dimension $n\le 3$,
 while \cite{matsubara:2010} provided the formula up to $O(\nu^2)$ ($\nu^2=N^{-1}$)
using 3- and 4-point correlation functions for $n=2$.
This study completes the asymptotic expansion formulas up to $O(N^{-1})$
using 3- and 4-point correlation functions for an arbitrary dimension $n$.
However, this does not only complete the existing research.
The discussions for arbitrary $n$ are informative and reveal the properties of the Minkowski functionals (e.g., Theorem \ref{thm:loop}).

In related work, \cite{chamandy-worsley-taylor-gosselin:2008} derived an approximation for the expected Euler characteristic of the excursion set of an isotropic central limit random field.
Their approach is based on a version of the saddle-point approximation, which is different from the Edgeworth-type expansion approach used in this study.
Our results are described in terms of the derivatives of $k$-point correlation functions,
and can be translated in terms of the higher-order spectra.
Another difference is that in \cite{chamandy-worsley-taylor-gosselin:2008} the threshold $x$ increases as the sample size $N$ increases, whereas $x$ of this study is assumed fixed.

The authors are preparing another cosmic paper \cite{matsubara-kuriki:2021},
where the perturbation formula for the Euler characteristic density up to $O(\nu^2)$ is derived using an alternative approach.
The final formulas are presented in terms of higher-order spectra.
It is confirmed that the formulas in the two studies are consistent.

This paper is organized as follows.
In Section \ref{sec:preliminary}, the Minkowski functional, and the Lipschitz-Killing curvature are defined, and (\ref{minkowski}) is proved.
The formula for the Euler characteristic density is then presented as the starting point of this study.
The main results are presented in Section \ref{sec:main}.
The isotropic $k$-point correlation functions are introduced, and then
the asymptotic expansion formula for the Euler characteristic density up to $O(N^{-1})$ is derived.
In addition, a class of the non-Gaussianity which cannot be captured by the Minkowski functional approach is identified.
In Section \ref{sec:example}, as an example, we consider an isotropic chi-square random field, which is a typical weakly-Gaussian random field when the degrees of freedom are large.
The isotropic chi-square random field belongs to a class of Gaussian related random fields whose Euler characteristic density is explicitly known.
We analytically and numerically confirmed the precision of the asymptotic expansion approximations.
Proofs of the main results are presented in Section \ref{sec:proofs}.
In the Appendix, the identities of the Hermite polynomial are proved (Section \ref{sec:Hermite}), and the regularity conditions for the asymptotic expansion are summarized (Section \ref{sec:asympt}).

\section{Preliminaries}
\label{sec:preliminary}

\subsection{Tube and Minkowski functional}
\label{sec:tube}

We begin with a quick review of the Minkowski functional.
Let $M$ be a bounded closed domain in $\R^n$.
For $u\in\R^n$, let $u_M$ be a point such that $\Vert u_M-u\Vert=\min_{w\in M}\Vert w-u\Vert$.
Note that $u_M$ exists but may not be unique.
A tube about $M$ with radius $r$ is defined by a set of points, the distance of which from $M$ is less than or equal to $r$:
\begin{equation}
\label{tube}
\Tube(M,r) = \bigl\{ u\in\R^n \mid \Vert u_M-u\Vert \le r \bigr\}.
\end{equation}
Then, the {\it critical radius} or {\it reach} is defined as
\[
 r_\mathrm{cri}(M) = \inf\bigl\{r\ge 0 \mid \mbox{$u_M$ is unique for all $u\in\Tube(M,r)$} \bigr\}.
\]
$r_\mathrm{cri}(M)$ is the maximum radius of the tube $\Tube(M,r)$ which does not have self-overlap (\cite[Section 2.2]{kuriki-takemura-taylor}).
$M$ is said to be of positive reach if $r_\mathrm{cri}(M)$ is strictly positive.
The classical Steiner formula states that when $M$ is a $C^2$-stratified manifold ($C^2$-piecewise smooth manifold \cite{takemura-kuriki:2002}) of positive reach, for all $r\in[0,r_\mathrm{cri}(M)]$, the volume of the tube (\ref{tube}) is expressed as a polynomial in $r$:
\begin{equation}
\label{steiner}
 \Vol_n(\Tube(M,r))
 = \sum_{j=0}^n \omega_{n-j}r^{n-j} \mathcal{L}_j(M)
 = \sum_{j=0}^n r^{j} \binom{n}{j} \mathcal{M}_j(M),
\end{equation}
where
$\Vol_n(\cdot)$ is the $n$-dimensional volume and
$\omega_j=\pi^{j/2}/\Gamma(j/2+1)$
is the volume of the unit ball in $\R^j$.
The Minkowski functional $\mathcal{M}_j(M)$ of $M$, and the Lipschitz-Killing curvature $\mathcal{L}_j(M)$ of $M$ are defined as the coefficients of the polynomial.
Note that $\mathcal{L}_j(M)$ is defined independently of the dimension $n$ of the ambient space.
There are variations of the definitions of the Minkowski functional.
The definition in (\ref{steiner}) is by \cite[Section 14.2]{schneider-weil:2008}.

Because of the Gauss-Bonnet theorem,
the Minkowski functional of the largest degree is proportional to the Euler characteristic of $M$:
\[
 \chi(M) = \mathcal{L}_0(M) = \mathcal{M}_n(M)/\omega_n.
\]

Throughout this study, it is assumed that the domain $T$ of the random field $X(t)$ is a $C^2$-stratified manifold of positive reach.
In the following, we prove (\ref{minkowski}).
Let $A(n,k)$ be the set of $k$-dimensional affine subspaces in $\R^n$.
Let $L\in A(n,n-k)$, and let $X|_L$ be the restriction of $X$ on $L$,
that is, a random field on $T\cap L$.
Because $X$ is isotropic, when $L$ is given, we have from (\ref{euler}),
\[
  \E[\chi((T\cap L)_x)]
 = \sum_{j=0}^{n-k} \mathcal{L}_j(T\cap L)\,\Xi_{j,N}(x).
\]
Note that $(T\cap L)_x=T_x\cap L$.
Let $\mu_{n,n-k}$ be the invariant measure over $A(n,n-k)$ normalized such that
$\mu_{n,k}(\{L\in A(n,k) \,|\, L\cap \B^n\ne\emptyset \}) = \omega_{n-k}$,
where $\B^n$ denotes the unit ball in $\R^n$.
Taking the integral $\int_{A(n,k)}\dd\mu_{n,n-k}(L)$,
by the generalized Crofton's theorem for a positive-reach set \cite{hug-schneider:2002}, we obtain
\[
 c_{n,0,k} \E[\mathcal{L}_{k}(T_x)]
 = \sum_{j=0}^{n-k} c_{n,j,k} \mathcal{L}_{k+j}(T)\,\Xi_{j,N}(x),
\]
where
$c_{n,j,k} = \bigl[\Gamma\bigl(\frac{k+1}{2}\bigr)\Gamma\bigl(\frac{k+j+1}{2}\bigr)\bigr]/\bigl[\Gamma\bigl(\frac{j+1}{2}\bigr)\Gamma\bigl(\frac{n+1}{2}\bigr)\bigr]$, which proves (\ref{minkowski}).

\subsection{Marginal distributions of $X$ and its derivatives}
\label{sec:marginal}

In this study, we deal with the central limit random field $X(t)=X_N(t)$ in (\ref{cl_field}) on $T\subset\R^n$.
We assume that $X(t)$ has a zero mean, unit variance, and a smooth sample path $t\mapsto X(t)$ in the following sense:
$X(t)$, $\nabla X(t)=(X_i(t))_{1\le i\le n}\in\R^n$, and $\nabla^2 X(t)=(X_{ij}(t))_{1\le i,j\le n}\in\Sym(n)$ (the set of $n\times n$ real symmetric matrices) exist and are continuous with respect to $t=(t^1,\ldots,t^n)$ a.s., where
\begin{equation}
\label{as_der} 
 X_i(t) = \frac{\partial}{\partial t^i}X(t), \qquad
 X_{ij}(t) = \frac{\partial^2}{\partial t^i\partial t^j}X(t).
\end{equation}
In addition, it is assumed that $X(\cdot)$ is isotropic.
That is, the arbitrary finite marginal distribution $\{X(t)\}_{t\in T'}$, where $T'\subset\R^n$ is a finite set, is invariant under the group of rigid motions of $t$.

This isotropic property implies that the marginal moment is independent of $t$.
Recall that we assumed $\E[X(t)]=0$ and $\E[X(t)^2]=1$.
The covariance function of an isotropic field is a function of the distance between two points:
\begin{equation}
\label{rho}
 \E[X(t_1)X(t_2)] = \E[Z_{(i)}(t_1) Z_{(i)}(t_2)] = \rho\bigl(\tfrac12\Vert t_1-t_2\Vert^2\bigr),
\end{equation}
where $Z_{(i)}$ was used in (\ref{cl_field}).
This covariance function is sufficiently smooth at $t_1=t_2$ by assuming the condition on $\rho$:
\begin{equation}
\label{rho_cond}
 \rho(0)=1,\ \rho'(0)<0, \mbox{ and } \dd^4 \rho(x)/\dd x^4 \mbox{ exists}.
\end{equation}
Under condition (\ref{rho_cond}),
$\partial^{4}\rho(\frac{1}{2}\Vert t_1-t_2\Vert^2)/
 \partial t_{1}^{i_1}\partial t_{1}^{j_1}
 \partial t_{2}^{i_2}\partial t_{2}^{j_2}$
exists, and 
there exist the mean square derivatives $X^*_i$ and $X^*_{ij}$ of $X$.
Their moments of order up to 2 are obtained by changing the derivatives and the expectation symbol $\E[\cdot]$
(\cite[Theorem 2.2.2]{adler:1981}).
For instance,
\begin{align}
\label{XX}
 \E[X^*_{ij}(t) X^*_{kl}(t)]
 = \frac{\partial^4}{\partial t_{1}^{i}\partial t_{1}^{j}\partial t_{2}^{k}\partial t_{2}^{l}} \E[X(t_1)X(t_2)] \Big|_{t_1=t_2=t}.
\end{align}
Moreover, (\ref{XX}) is equivalent to $\E[X_{ij}(t) X_{kl}(t)]$ when the a.s.\ derivatives (\ref{as_der}) exist.
In this manner, we obtain the moments of $(X(t),\nabla X(t),\nabla^2 X(t))$ of order up to 2 are obtained as follows:
$\E[X_i(t)] = \E[X_{ij}(t)] = \E[X_{i}(t)X(t)] = \E[X_{ij}(t)X_{k}(t)] = 0$,
$\E[X_i(t) X_j(t)]=-\rho'(0)\delta_{ij}$,
$\E[X_{ij}(t)X(t)]=\rho'(0)\delta_{ij}$, and
\[
 \E[X_{ij}(t)X_{kl}(t)]=\rho''(0) (\delta_{ik}\delta_{jl}+\delta_{il}\delta_{jk}+\delta_{ij}\delta_{kl}),
\]
where $\delta_{ij}$ is the Kronecker delta.
In particular, $\nabla X(t)$ is uncorrelated with $X(t)$ for each fixed $t$.
We change the variable from $\nabla^2 X(t)$ to
\begin{equation*}
 R(t) = (R_{ij}(t))_{1\le i,j\le n} = \nabla^2 X(t) + \gamma X(t) I_n,
 \quad\gamma=-\rho'(0).
\end{equation*}
Then, $R(t)$ is uncorrelated with $X(t)$ and $\nabla X(t)$ for each fixed $t$.
This simplifies the entire manipulation process.
$R(t)$ has a zero mean and a covariance structure
\begin{align}
\label{ERR}
\E[R_{ij}(t) R_{kl}(t)]
= \alpha \frac{1}{2} (\delta_{ik}\delta_{jl}+\delta_{il}\delta_{jk}) + \beta \delta_{ij}\delta_{kl},
\end{align}
where $\alpha = 2\rho''(0)$ and $\beta = \rho''(0)-\rho'(0)^2$.
As an $\binom{n+1}{2}\times\binom{n+1}{2}$ covariance matrix,
(\ref{ERR}) is nonnegative definite 
if and only if $\alpha\ge 0$ and $\alpha+n\beta\ge 0$, and is positive definite if and only if
\begin{equation}
\label{rho12}
 \alpha>0,\,\alpha+n\beta>0, \quad\mbox{or equivalently},\quad
 \rho''(0) > \frac{n}{n+2}\rho'(0)^2 > 0.
\end{equation}
In this study, we assume (\ref{rho12}).
This is satisfied when $(X(t),\nabla X(t),\nabla^2 X(t))\in\R^{1+n+\binom{n+1}{2}}$ has a probability density function.
This assumption is sufficient for our motivative applications in cosmology (\cite{matsubara:2003}).

In the limit $N\to\infty$, for each fixed $t$, $X(t)$, $\nabla X(t)$, and $R(t)$ are independently distributed as Gaussian distributions $X(t)\sim\mathcal{N}(0,1)$, $\nabla X(t)\sim \mathcal{N}_n(0,\gamma I_n)$, and $R(t)\in\Sym(n)$
is a zero-mean Gaussian random matrix with a covariance structure (\ref{ERR}).
This limiting Gaussian random matrix $R(t)$ is referred to as the Gaussian orthogonal invariant matrix (\cite{cheng-schwartzman:2018})
and has a probability density function
\begin{align*}
p^0_{R}(R)
\propto
\exp\Bigl\{-\frac{1}{2\alpha}\tr\bigl((R-\tfrac{1}{n}\tr(R) I_n)^2\bigr) - \frac{1}{2 n(\alpha+n\beta)}\tr(R)^2 \Bigr\}
\end{align*}
with respect to $\dd R=\prod_{1\le i\le j\le n} \dd R_{ij}$.

\cite{cheng-schwartzman:2018} proved that in the boundary case $\rho''(0)=(n/(n+2))\rho'(0)^2>0$, there exists an isotropic Gaussian random field.
For non-Gaussian random fields in the boundary case, nothing seems to be known.

\subsection{Euler characteristic density}

In the following, let $V(t)=(V_i(t))_{1\le i\le n} =\nabla X(t)$.
The probability density function of $(X(t),V(t),R(t))$ with respect to the Lebesgue measure $\dd X\dd V\dd R$, $\dd V=\prod_{i=1}^n \dd V_i$, $\dd R= \prod_{1\le i\le j\le n} \dd R_{ij}$ is denoted by $p_N(X,V,R)$ if exists.
This is irrespective of point $t$.

Morse's theorem is a fundamental tool for counting the Euler characteristics of a set.
The following is a result of the Kac-Rice formula, which is the integral form of Morse's theorem.
\begin{proposition}[\cite{adler-taylor:2007}]
\label{prop:kac-rice}
Suppose that the regularity conditions in Theorem 11.2.1 of \cite{adler-taylor:2007} with $f(t)$ and $g(t)$ replaced by $\nabla X(t)$ and $(X(t),\nabla^2 X(t))$, respectively,
are satisfied.
Then, the Euler characteristic density in (\ref{euler}) and (\ref{minkowski}) is
\begin{equation}
\label{Xin}
 \Xi_{n,N}(x) =
 \int_{x}^\infty \biggl[ \int_{\Sym(n)} \det\bigl(-R +\gamma x' I_n\bigr) p_N(x',0,R) \dd R \biggr] \dd x'.
\end{equation}
\end{proposition}

When the random field $X(\cdot)$ is Gaussian, the conditions for Proposition \ref{prop:kac-rice} are simplified (\cite[Corollary 11.2.2]{adler-taylor:2007}).
We use the formula (\ref{Xin}) as the starting point of the analysis.
In the limit $N\to\infty$, $p_{\infty}(x,0,R)=\phi(x)p^0_{V}(0)p^0_{R}(R)$, where
\begin{equation}
\label{gauss_pdf}
 \phi(x) = \frac{1}{\sqrt{2\pi}}e^{-x^2/2}
\end{equation}
is the probability density function of the standard Gaussian distribution $\mathcal{N}(0,1)$ and $p^0_{V}(0)=(2\pi\gamma)^{-n/2}$ is the probability density function of $V$ evaluated at $V=0$.
Then, we have the well-known result
\begin{equation}
\label{Xi_infty}
 \Xi_{n,\infty}(x)=\frac{\gamma^{n/2}}{(2\pi)^{n/2}}\phi(x) H_{n-1}(x),
\end{equation}
where
\begin{equation}
\label{hermite}
 H_k(x) = \phi(x)^{-1}
 \Bigl(-\frac{\dd}{\dd x}\Bigr)^k\phi(x)
\end{equation}
is the Hermite polynomial 
(e.g., \cite{tomita:1986},\cite{adler-taylor:2007}).

The primary purpose of this study is to derive the asymptotic expansion formula of $\Xi_{n,N}(x)$ around $N=\infty$.

\section{Main results}
\label{sec:main}

This section presents the main theorems of this study.
The statements are described in terms of the isotropic cumulants introduced below.

\subsection{Isotropic cumulants and their derivatives}
\label{sec:iso_cumulants}

If a function $f(t_1,\ldots,t_k)$, $t_i\in\R^n$, is isotropic, then $f$ is a function of $t_i$'s through the distances between them: $\Vert t_i-t_j\Vert$, $1\le i<j\le k$.
Then, the $k$-point correlation function of the isotropic central limit random field $X(\cdot)$ in (\ref{cl_field}) is written as
\begin{equation}
\label{cum}
 \cum(X(t_1),\ldots,X(t_k)) 
= N^{-\frac{1}{2}(k-2)} \kappa^{(k)}(x_{12},x_{13},\ldots,x_{k-1,k}),\quad x_{ab}=\tfrac12 \Vert t_a-t_b\Vert^2,
\end{equation}
where $\kappa^{(k)}$ denotes the $k$th order cumulant of $Z_{(i)}(\cdot)$ in (\ref{cl_field}).
Note that the cumulant of a random vector $(Y_1,\ldots,Y_k)$ is defined by
\[
 \cum(Y_1,\ldots,Y_k) = \sum (-1)^\ell (\ell-1)!\,
 \E\bigl[\textstyle{\prod_{i\in I_1}}Y_i\bigr]\cdots\E\bigl[\textstyle{\prod_{i\in I_\ell}}Y_i\bigr],
\]
where the summation runs over all possible set partitions of $\{1,\ldots,k\}$ such that $I_1\sqcup\cdots\sqcup I_\ell = \{1,\ldots,k\}$
(\cite{mccullagh:1987}).
When and only when $k$-dimensional marginal $(X(t_1),\ldots,X(t_k))$ has a $k$th order moment, the $k$-point correlation function exists.
The 2-point correlation function $\kappa^{(2)}(\cdot)$ is the covariance function $\rho(\cdot)$ in (\ref{rho}).
The 3- and 4-point correlation functions are as follows:
\begin{align*}
& \cum(X(t_1),X(t_2),X(t_3)) = \E[X(t_1)X(t_2)X(t_3)]
 = N^{-\frac{1}{2}} \kappa^{(3)}(x_{12},x_{13},x_{23}), \\
& \cum(X(t_1),X(t_2),X(t_3),X(t_4)) \\
&\qquad = \E[X(t_1)X(t_2)X(t_3)X(t_4)] -\E[X(t_1)X(t_2)]\E[X(t_3)X(t_4)][3] \\
&\qquad = N^{-1} \kappa^{(4)}(x_{12},x_{13},x_{14},x_{23},x_{24},x_{34}),
\end{align*}
where ``$[3]$'' represents the three symmetric terms.
Note that $\kappa^{(3)}(x_{12},x_{13},x_{23})$ is symmetric in its arguments, but $\kappa^{(k)}$ ($k\ge 4$) are not symmetric.

We assume the smoothness of $\kappa^{(k)}$ in (\ref{cum})
as a generalization of (\ref{rho_cond}): 
\begin{equation}
\label{cum_cond}
 \frac{\partial^{2k}\kappa^{(k)}(x_{12},x_{13},\ldots,x_{k-1,k})}
  {\prod_{1\le a<b\le k} (\partial x_{ab})^{n_{ab}}}\ \ %
 \mbox{exists for $\sum n_{ab}=2k$, $n_{ab}\le 4$},
\end{equation}
from which, it is proved that
\begin{equation}
\label{cum_cond0}
 \frac{\partial^{2k}\kappa^{(k)}(x_{12},x_{13},\ldots,x_{k-1,k})}
 {\partial t_{1}^{i_1}\partial t_{1}^{j_1}\cdots
 \partial t_{k}^{i_k}\partial t_{k}^{j_k}},\ \ x_{ab}=\tfrac12 \Vert t_a-t_b\Vert^2 \ \ \mbox{exists}.
\end{equation}
Under this condition, the $k$th order cumulants of $(X(t),\nabla X(t),\nabla^2 X(t))$ are obtained by changing the derivatives and symbol of the cumulant.
For instance,
\begin{equation}
\label{cum_example}
 \cum(X_{i}(t),X_{j}(t),X_{kl}(t))
 = \frac{\partial^4}{\partial t_1^i\partial t_2^j\partial t_3^k\partial t_3^l}
 \cum(X(t_1),X(t_2),X(t_3)) \Big|_{t_1=t_2=t_3=t}.
\end{equation}
(The proof is identical to that in Section \ref{sec:marginal}.
The mean square derivatives $X^*_i$ and $X^*_{ij}$ that exist under (\ref{rho_cond}) satisfy (\ref{cum_example}) under (\ref{cum_cond0}).)

To state the main theorems, the following notations are used:
\begin{equation}
\label{kappaN} 
\kappa^{(k)}_{(a_1 b_1),\ldots,(a_K b_K)}(0)
 = \kappa^{(k)}_E(0)
 = \Biggl(\prod_{(a,b)\in E}\Bigl(\frac{\partial}{\partial x_{ab}}\Bigr)\Biggr) \kappa^{(k)}\bigl((x_{ab})_{1\le a<b\le k}\bigr)\Big|_{x_{12}=\cdots=x_{k-1,k}=0}
\end{equation}
with
$E = \{(a_1,b_1),\ldots,(a_K,b_K)\}$.
Here, an undirected graph $(V,E)$ with vertex set $V=\{1,\ldots,k\}$ and edge set $E$ is considered.
Note that the edge set $E$ is a multiset that allows multiple edges connecting a pair of two vertices.
Then, a \textit{diagram} is defined as the undirected graph $(V,E)$ without the information of labeling vertices.
In addition, isolated vertices are omitted from the diagram.
That is,
two diagrams are identical if their undirected graphs are identical due to a suitable relabeling of the vertices.
As shown in Figure \ref{fig:loop+tree}, some diagrams contain cycles, while others do not.

\begin{figure}[ht]
\begin{center}
\medskip
\begin{tabular}{ccccccc}
 \scalebox{0.5}{\includegraphics{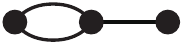}}
 &&
 \scalebox{0.5}{\includegraphics{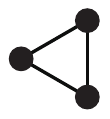}}
 &&
 \scalebox{0.5}{\includegraphics{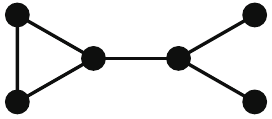}}
 &&
 \scalebox{0.5}{\includegraphics{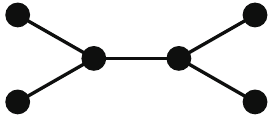}}
\\
(a) && (b) && (c) && (d)
\end{tabular}
\caption{Diagrams with cycle (a--c) and without cycle (d). \newline
\small
(a) $\kappa^{(3)}_{(12),(12),(13)}(0)$, \quad
(b) $\kappa^{(3)}_{(12),(13),(23)}(0)$, \newline
(c) $\kappa^{(6)}_{(12),(13),(14),(23),(45),(46)}(0)$, \quad
(d) $\kappa^{(6)}_{(12),(13),(14),(45),(46)}(0)$.
}
\label{fig:loop+tree}
\end{center}
\end{figure}

Table \ref{tab:derivatives} lists the cycle-free derivatives (\ref{kappaN}) for $k\le 4$ and their abbreviations.
We prepared this table for the statement of the main theorem in the next section.
The \textit{multiplicity} is the number of undirected graphs that are identical by way of relabeling the vertices.

\begin{table}[ht]
\caption{Derivatives of cumulant functions $\kappa^{(k)}$ with cycle-free diagram ($k=2,3,4$).}
\label{tab:derivatives}
\begin{center}
\begin{tabular}{ccccc}
\hline
abbreviation & \parbox[c][6ex][c]{12ex}{representative \\ derivative} & multiplicity & diagram \\ \hline
 $-\gamma$ & $\rho'(0)$ & 1 &
 $\rule{0pt}{14pt}$
 \scalebox{0.625}{\includegraphics{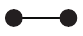}} \\
\hline
 $\ko$         & $\kappa^{(3)}(0)$ & 1 & \\
 $\ki$         & $\kappa^{(3)}_{(12)}(0)$ & 3 &
 \scalebox{0.625}{\includegraphics{k1_s.pdf}} \\
 $\kii$        & $\kappa^{(3)}_{(12),(13)}(0)$ & 3 &
 \scalebox{0.625}{\includegraphics{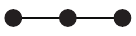}} \\
\hline
 $\kko$         & $\kappa^{(4)}(0)$ & 1 & \\
 $\kki$         & $\kappa^{(4)}_{(12)}(0)$ & 6 &
 \scalebox{0.625}{\includegraphics{k1_s.pdf}} \\
 $\kkiia$       & $\kappa^{(4)}_{(12),(13)}(0)$ & 12 &
 \scalebox{0.625}{\includegraphics{k11_s.pdf}} \\
 $\kkiiaa$      & $\kappa^{(4)}_{(12),(34)}(0)$ & 3 &
 \scalebox{0.625}{\includegraphics{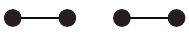}} \\
 $\kkiiid$     & $\kappa^{(4)}_{(12),(13),(14)}(0)$ & 4 &
 \scalebox{0.625}{\includegraphics{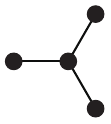}} \\
 $\kkiiia$     & $\kappa^{(4)}_{(12),(13),(24)}(0)$ & 12 &
 \scalebox{0.625}{\includegraphics{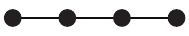}} \\
\hline
\end{tabular}
\end{center}
\end{table}

\subsection{Asymptotic expansion of the Euler characteristic density}

The two theorems are stated here as the main results.
These are proved in the following sections.

\begin{assumption}
\label{assumptions}
{\rm (i)}
The covariance function $\rho$ in (\ref{rho}) and the third and fourth order cumulant functions $\kappa^{(k)}$ in (\ref{cum}) for $k=3,4$ exist, and satisfy (\ref{rho_cond}), and (\ref{cum_cond}) or (\ref{cum_cond0}), respectively.

{\rm (ii)}
The probability density function $p_N(X,V,R)$ of $(X(t),V(t),R(t))$ for $t$ fixed exists for $N\ge 1$, and is bounded for some $N$.
$(X(t),V(t),R(t))$ has a moment of order $\binom{n+2}{2}+1$ $(\ge 4)$
under $p_1$.
\end{assumption}

\begin{theorem}
\label{thm:expansion}
Under Assumption \ref{assumptions},
as $N\to\infty$, the Euler characteristic density (\ref{Xin}) is expanded as
\begin{align*}
\Xi_{n,N}(x)
=& \frac{\gamma^{n/2}}{(2\pi)^{n/2}} \phi(x) \Bigl( H_{n-1}(x) + \frac{1}{\sqrt{N}} \Delta_{1,n}(x) + \frac{1}{N} \Delta_{2,n}(x) \Bigr) + o(N^{-1})
\end{align*}
uniformly in $x$, where
\begin{align*}
\Delta_{1,n}(x) =
&  \tfrac{1}{2} \gamma^{-2} \kii n (n-1) H_{n-2}(x)
 - \tfrac{1}{2} \gamma^{-1} \ki n H_{n}(x)
 + \tfrac{1}{6} \ko H_{n+2}(x),
\end{align*}
\begin{align*}
\Delta_{2,n}(x) =
& \Bigl( -\tfrac{1}{6}\gamma^{-3} (3\kkiiia + \kkiiid)
 +\tfrac{1}{8}\gamma^{-4} \kii^2 (n-7) \Bigr) n(n-1)(n-2) H_{n-3}(x) \nonumber \\
& + \Bigl( \tfrac{1}{8}\gamma^{-2} \bigl(\kkiiaa (n-2) +4\kkiia (n-1)\bigr) \nonumber \\
&\quad -\tfrac{1}{4}\gamma^{-3} \ki\kii (n-1)(n-4) \Bigr) n H_{n-1}(x) \nonumber \\
& + \Bigl( - \tfrac{1}{4}\gamma^{-1} \kki 
 +\tfrac{1}{24}\gamma^{-2} \bigl(3\ki^2 (n-2) + 2\ko\kii (n-1)\bigr) \Bigr) n H_{n+1}(x) \nonumber \\
& + \Bigl( \tfrac{1}{24} \ko -\tfrac{1}{12}\gamma^{-1} \ko\ki n \Bigr) H_{n+3}(x)
 +\tfrac{1}{72} \ko^2 H_{n+5}(x). \nonumber
\end{align*}
Here, the symbols
$\gamma,\ko,\ki,\kii,\kko,\kki,\kkiia,\kkiiaa,\kkiiid$ and $\kkiiia$ are defined in Table \ref{tab:derivatives}.
\end{theorem}

Note that, in principle, the asymptotic expansion of $\Xi_{n,N}(x)$ can be obtained up to an arbitrary order in $N$ by modifying Assumption \ref{assumptions}.
In $\Delta_{1,n}(x)$ and $\Delta_{2,n}(x)$, 
derivatives with cycles in their diagrams, such as Figures \ref{fig:loop+tree} (a) and (b), do not appear in the expansions.
Conversely, all types of derivatives of $\kappa^{(3)}$ and $\kappa^{(4)}$ with cycle-free diagrams listed in Table \ref{tab:derivatives} appear.
The first half of this observation holds true for an arbitrary order as follows.

\begin{theorem}
\label{thm:loop}
Suppose that the diagram of the derivative 
$\kappa^{(k)}_{(a_1 b_1),\ldots,(a_K b_K)}(0)$ in (\ref{kappaN})
has cycles of length greater than or equal to $2$.
Then, it does not appear in the asymptotic expansion of $\Xi_{n,N}(x)$ even if when expanded to $O\bigl(N^{-\frac{1}{2}(k-2)}\bigr)$.
\end{theorem}

We conjecture that the non-Gaussianity (that is, $\kappa^{(k)}$ and its derivatives) captured by the Minkowski functional approach is characterized by the presence or absence of cycles in the diagrams.

\section{Chi-square random field}
\label{sec:example}

Here we consider a chi-square random field defined by the squared sum of the independent copies of a Gaussian random field.

Let $Y(t)$  be a $C^2$-Gaussian random field
 on $T\subset\R^n$ with zero mean and covariance function $\E[Y(s) Y(t)] = \rhoY(\frac{1}{2}\Vert s-t\Vert^2)$ such that $\rhoY(0)=1$, $\rhoY'(0)=-g<0$,
$\rhoY''(0)>(n/(n+2))\rhoY'(0)^2$, and $\dd^4\rhoY(x)/\dd x^4$ exists.
For example, $\rhoY(x)=e^{-g x}$, $g>0$.
Then, define
\begin{equation*}
 X(t) = X_N(t) = \frac{1}{\sqrt{2N}} \sum_{i=1}^N \bigl(Y_{(i)}(t)^2-1\bigr),
\end{equation*}
where $Y_{(i)}(\cdot)$ are i.i.d.\ copies of $Y(\cdot)$.

The chi-square random field has been well studied as one of the simplest non-Gaussian random fields (\cite[Section 7.1]{adler:1981}, \cite{worsley:1994}, \cite{matsubara-yokoyama:1996}, \cite{taylor:2006}).
Here, we verify that Assumption \ref{assumptions} is satisfied.

According to \cite[Lemma 3.2]{worsley:1994}, $\nabla X=\nabla X(t)$ and $\nabla^2 X=\nabla^2 X(t)$ can be decomposed as
\begin{equation}
\label{decompositon}
\nabla X = 2 g^{\frac{1}{2}} X^{\frac{1}{2}} U, \qquad
\nabla^2 X = 2 g(P + U U^\top - X I_n - X^{\frac{1}{2}}R),
\end{equation}
where 
$X=X(t)\sim\chi^2_N$, $U\sim\mathcal{N}_n(0,I_n)$, $P\sim\mathcal{W}_{n\times n}(N-1,I_n)$, and $R\in\Sym(n)$ a Gaussian orthogonal invariant matrix with parameters $\alpha=2\rhoY''(0)/g^2$ and $\beta=\rhoY''(0)/g^2-1$ in (\ref{ERR}),
are independently distributed.
From (\ref{decompositon}), $(X,\nabla X,\nabla^2 X)$ has moments of arbitrary order.
In addition, the conditional probability density of $(X,\nabla X,\nabla^2 X)$ given $R$ is obtained as
\begin{align*}
 p(X,\nabla X,\nabla^2 X|R) \propto
 & \det(P)^{\frac{1}{2}(N-n-2)}
 e^{-\frac{1}{2}\tr(P)} e^{-\frac{1}{8 g X}\Vert\nabla X\Vert^2} X^{\frac{1}{2}(N-n)-1}e^{-X/2} \1(P\succ 0),
\end{align*}
where
\[
 P = (2 g)^{-1}\nabla^2 X - (4 g)^{-1}X^{-1}(\nabla X)(\nabla X)^\top + X I_n + X^{\frac{1}{2}}R,
\]
and $P\succ 0$ indicates that $P$ is positive definite.
This conditional density is continuous and bounded above when $N$ is large.
Hence, so is the unconditional density $p(X,\nabla X,\nabla^2 X)=\E^R[p(X,\nabla X,\nabla^2 X|R)]$, and Assumption \ref{assumptions} (ii) is satisfied.

The $k$-point correlation function $\kappa^{(k)}$ is the same as the cumulant of $Z(t)=(Y(t)^2-1)/\sqrt{2}$ as follows.
Let $x_{ab}=\frac{1}{2}\Vert t_a-t_b\Vert^2$.
\begin{align*}
 & \rho(x_{12}) = \kappa^{(2)}(x_{12}) = \cum(Z(t_1),Z(t_2))
 = \rhoY(x_{12})^2, \\
 & \kappa^{(3)}(x_{12},x_{13},x_{23}) = \cum(Z(t_1),Z(t_2),Z(t_3))
  = 2^{3/2}\rhoY(x_{12})\rhoY(x_{13})\rhoY(x_{23}),
\end{align*}
and
\begin{align*}
 & \kappa^{(4)}(x_{12},x_{13},\ldots,x_{34}) =
 \cum(Z(t_1),Z(t_2),Z(t_3),Z(t_4)) \\
 & = 4\bigl[\rhoY(x_{13})\rhoY(x_{14})\rhoY(x_{23})\rhoY(x_{24})
 + \rhoY(x_{12})\rhoY(x_{14})\rhoY(x_{23})\rhoY(x_{34}) \\
 &\quad + \rhoY(x_{12})\rhoY(x_{13})\rhoY(x_{24})\rhoY(x_{34})\bigr].
\end{align*}
We see that $\kappa^{(k)}$, $k=2,3,4$, defined above satisfy the requirements (\ref{rho_cond}) and (\ref{cum_cond}).
Hence, Assumption \ref{assumptions} (i) is satisfied.

The cumulants listed in Table \ref{tab:derivatives} are as follows:
\begin{align*}
& \rho(0) = 1, \quad 
-\gamma = \rho'(0) = -2 g, \quad
\ko = 2\sqrt{2}, \quad
\ki = -\sqrt{2}\gamma, \quad
\kii = \gamma^2/\sqrt{2}, \\
& \kko = 12, \quad
\kki = -4\gamma, \quad
\kkiia = \gamma^2, \quad
\kkiiaa = 2\gamma^2, \quad
\kkiiid = 0, \quad
\kkiiia = -\gamma^3/2.
\end{align*}
$\Delta_{1,n}(x)$ and $\Delta_{2,n}(x)$ in Theorem \ref{thm:expansion} are
\begin{equation}
\label{Delta12}
\begin{aligned}
\Delta_{1,n}(x) =&  
\sqrt{2}\bigl(\tfrac{1}{4}n(n-1) H_{n-2}(x) +\tfrac{1}{2}n H_{n}(x) +\tfrac{1}{3}H_{n+2}(x)\bigr), \\
\Delta_{2,n}(x) =&
 \tfrac{1}{16}n(n-1)(n-2)(n-3) H_{n-3}(x) +\tfrac{1}{4} n^2(n-2) H_{n-1}(x) \\
& +\tfrac{1}{12}n(5n+4) H_{n+1}(x) +\tfrac{1}{6}(2n+3) H_{n+3}(x) +\tfrac{1}{9}H_{n+5}(x).
\end{aligned}
\end{equation}

The exact formula for the isotropic chi-square random field was obtained by \cite[Theorem 3.5]{worsley:1994} as follows:
\begin{equation}
\label{Xin_chi}
\Xi_{n,N}(x) = \frac{g^{n/2}}{(2\pi)^{n/2}} \frac{1}{2^{N/2-1}\Gamma(N/2)} H_{n-1}^{(N-1)}(\sqrt{y}) e^{-y/2}, \quad y=N+x\sqrt{2N},
\end{equation}
where
\begin{equation*}
 H^{(N)}_n(x) = e^{x^2/2} \Bigl(-\frac{\dd}{\dd x}\Bigr)^n \bigl(x^N e^{-x^2/2}\bigr).
\end{equation*}
This formula was also obtained as a special case of the Euler characteristic density of the Gaussian related random fields (\cite{adler-taylor:2007,adler-taylor:2011},\cite{panigrahi-etal:2019}).

Through direct calculations using a generating function (not using Theorem \ref{thm:expansion}), the following proposition can be shown.
The proof is provided at the end of this section.
\begin{proposition}
\label{prop:generating_fn}
$\Xi_{n,N}(x)$ in (\ref{Xin_chi}) is expanded as
\begin{equation*}
\Xi_{n,N}(x) = \frac{\gamma^{n/2}}{(2\pi)^{n/2}} \phi(x) \Bigl( H_{n-1}(x) + \frac{1}{\sqrt{N}} \Delta_{1,n}(x) + \frac{1}{N} \Delta_{2,n}(x) \Bigr) + o(N^{-1})
\end{equation*}
as $N\to\infty$, 
where $\Delta_{1,n}(x)$ and $\Delta_{2,n}(x)$ are defined in (\ref{Delta12}). 
\end{proposition}

\begin{figure}[h]
\begin{center}
\scalebox{0.7}{\includegraphics{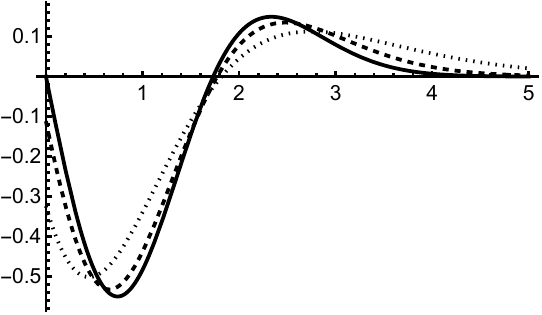}}
\caption{EC density for a chi-square random field on $\R^4$ of $N$ degrees of freedom.
(dotted: $N=10$, dashed: $N=100$, solid: $N=\infty$.)}
\label{fig:n4-m10_100_infty}
\end{center}
\end{figure}

Figure \ref{fig:n4-m10_100_infty} shows the Euler characteristic densities of the chi-square random field on $\R^4$ when the degrees of freedom $N$ is 10, 100, and $\infty$.
The curve converges to the limiting Gaussian curve as $N$ increases.
 
\begin{figure}[h]
\begin{center}
\scalebox{0.7}{\includegraphics{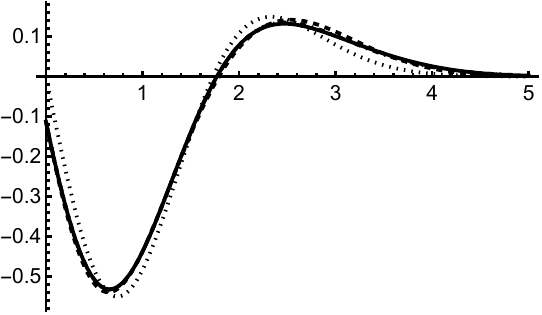}}
\caption{EC density for a chi-square random field on $\R^4$ of $100$ degrees of freedom and its approximations. \newline
 (dot-dashed: true curve, dotted: Gaussian approx., dashed: 1st approx., solid: 2nd approx.)}
\label{fig:n4-m100}
\end{center}
\end{figure}

Figure \ref{fig:n4-m100} shows the Euler characteristic density $\Xi_{n,N}(x)$ ($n=4$) of the chi-square random field and its Gaussian approximation: $(2\pi)^{-n/2}\phi(x) H_{n-1}(x)$, the 1st approximation: $(2\pi)^{-n/2}\phi(x)\bigl( H_{n-1}(x) + \Delta_{1,n}(x)/\sqrt{N} \bigr)$, and the 2nd approximation:
$(2\pi)^{-n/2} \phi(x) \bigl( H_{n-1}(x) + \Delta_{1,n}(x)/\sqrt{N} + \Delta_{2,n}(x)/N \bigr)$.

\begin{figure}[h]
\begin{center}
\scalebox{0.7}{\includegraphics{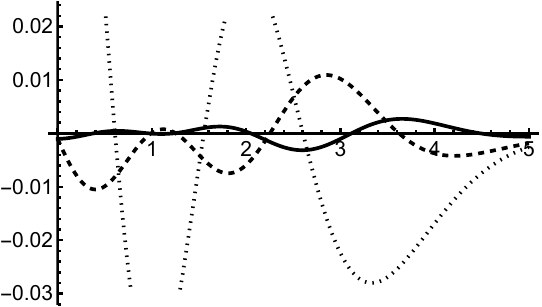}}
\caption{Approximation error of EC density for a chi-square random field on $\R^4$ of $100$ degrees of freedom. \newline
 (dotted: Gaussian approx., dashed: 1st approx., solid: 2nd approx.)}
\label{fig:n4-m100-diff}
\end{center}
\end{figure}

In Figure \ref{fig:n4-m100}, the four curves are too close to distinguish.
Figure \ref{fig:n4-m100-diff} shows the difference between the three approximations with respect to the true curve $\Xi_{n,N}(x)$.
We see that the 1st approximation is more accurate than the Gaussian approximation, and that the 2nd approximation is more accurate than the 1st approximation, as expected.

\begin{proof}[Proof of Proposition \ref{prop:generating_fn}]
We prove that
\begin{equation}
\label{der_chi_density}
 \frac{2^{-n/2}}{2^{N/2-1}\Gamma(N/2)} H_{n-1}^{(N-1)}(\sqrt{y}) e^{-y/2}, \quad y=N+x\sqrt{2N},
\end{equation}
is expanded as
$\phi(x) \bigl( H_{n-1}(x) + \Delta_{1,n}(x)/\sqrt{N} + \Delta_{2,n}(x)/N \bigr) + o(N^{-1})$,
where $\Delta_{1,n}(x)$ and $\Delta_{2,n}(x)$ are given in (\ref{Delta12}). 

By multiplying (\ref{der_chi_density}) by $z^{n-1}/(n-1)!$,
and taking the summation over $n\ge 1$, 
the generating function of (\ref{der_chi_density}) is obtained as
\begin{equation}
\label{gf}
\frac{1}{2^{(N-1)/2}\Gamma(N/2)}(-z/\sqrt{2}+\sqrt{y})^{N-1} e^{-\frac{1}{2}(-z/\sqrt{2}+\sqrt{y})^2}, \quad y=N+x\sqrt{2N}.
\end{equation}
The generating function (\ref{gf}) is expanded around $N=\infty$ as
\[
 \varphi_0(z,x)\bigl(1+p_1(z,x)/\sqrt{N}+p_2(z,x)/N \bigr) + o(N^{-1}),
\]
where
\[
 \varphi_0(z,x)=e^{-\frac{1}{2} (x-z)^2}/\sqrt{2\pi},
\]
and
\begin{align*}
 p_1(z,x) =& \bigl( \tfrac{1}{3}x(2x^2-3) -(x-1)(x+1)z +\tfrac{1}{2}x z^2 -\tfrac{1}{6}z^3 \bigr)/\sqrt{2}, \nonumber \\
 p_2(z,x) =&
 \tfrac{1}{36} (4x^6-30x^4+27x^2-6) -\tfrac{1}{12} (x-2)x(x+2) (4x^2-3) z \\
& +\tfrac{1}{12} (5x^4-15x^2+6) z^2
-\tfrac{1}{36} x(11x^2-21) z^3 \\
& +\tfrac{7}{48} (x-1)(x+1) z^4
-\tfrac{1}{24} x z^5
+\tfrac{1}{144} z^6.
\end{align*}
We select the coefficient of the term $z^{n-1}/(n-1)!$.
Because $e^{z x-z^2/2}$ is the generating function of the Hermite polynomial,
the coefficient of the term $z^{n-1}/(n-1)!$ in $\varphi(z,x) z^a$ is
\[
 \phi(x)H_{n-a-1}(x) (n-1)_a, \quad (n-1)_a = (n-1)(n-2)\cdots (n-a), 
\]
where $\phi(x)=\varphi(0,x)$ denotes the probability density function of $\mathcal{N}(0,1)$.
Based on this term-rewriting rule,
we obtain the two terms $\Delta_{1,n}(x)$ and $\Delta_{2,n}(x)$ in terms of the polynomials in $x$ and the Hermite polynomials in $x$.
By applying the three-term relation
\begin{equation}
\label{three-term}
 x H_{k}(x) = H_{k+1}(x) + k H_{k-1}(x),
\end{equation}
we have the expressions for $\Delta_{1,n}(x)$ and $\Delta_{2,n}(x)$ in terms of the Hermite polynomials only.
\end{proof}

\section{Proofs of the main results}
\label{sec:proofs}

In this section, we prove the main theorems.
The outline is as follows.
We first describe the characteristic function of $(X(t),\nabla X(t),\nabla^2 X(t))$ using the 3- and 4-point correlation functions of $Z_{(i)}(\cdot)$ in (\ref{cl_field})  (Section \ref{sec:iso_cgf}).
The resulting characteristic function is modified into the conditional characteristic function of $(X(t),\nabla^2 X(t))$ when $\nabla X(t)=0$ is given.
Then, the integral (\ref{Xin}) is obtained by taking derivatives of the conditional characteristic function (Section \ref{sec:proof_expansion}).
The validity of the asymptotic expansion is proved separately. 
Theorem \ref{thm:loop} is proved in Section \ref{sec:undetectable}.

\subsection{Isotropic cumulant generating function}
\label{sec:iso_cgf}

The objective function $\Xi_{n,N}(x)$ is an expectation with respect the distribution of $(X,V,R)$, $V=\nabla X$, $R=\nabla^2 X+\gamma X I_n$.
The index $t$ is supposed to be fixed and omitted. 
To evaluate this, we first identify the characteristic function of $(X,\nabla X,\nabla^2 X)$.
We introduce parameters
$T = (\tau_i)\in\R^n$ and $\Theta = (\theta_{ij})\in\Sym(n)$ such that
\begin{equation}
\label{Theta}
 \theta_{ij} = \frac{1+\delta_{ij}}{2}\tau_{ij} \quad (i\le j).
\end{equation}
The characteristic function of the $1+n+n(n+1)/2=\binom{n+2}{2}$ dimensional random variables $(X,\nabla X,\nabla^2 X)$ is
\begin{align}
\label{mgf0}
\psi_N(s,T,\Theta)
 =& \E\Bigl[ e^{\iii (s X+\sum_i \tau_i X_i+\sum_{i\le j} \tau_{ij}X_{ij})} \Bigr]
 = \E\Bigl[ e^{\iii(s X+\langle T,\nabla X\rangle+\tr(\Theta \nabla^2 X))} \Bigr].
\end{align}

When $(X,\nabla X,\nabla^2 X)$ has the $k_1$th moments,
the cumulant generating function $\log\psi_N(s,T,\Theta)$ has the following Taylor series
\[
 \log\psi_N(s,T,\Theta) = \sum_{k=2}^{k_1} \sum_{u+v+w=k}\frac{N^{-\frac{1}{2}(k-2)} \iii^k}{u!\,v!\,w!} K_{u,v,w}(s,T,\Theta) + o\bigl(N^{-\frac{1}{2}(k_1-2)}\bigr),
\]
where
\begin{equation}
\label{Kuvw}
\begin{aligned}
& N^{-\frac{1}{2}(k-2)} K_{u,v,w} (s,T,\Theta) \\
& = \cum\bigl(\underbrace{s X,\ldots,s X}_u,\underbrace{\langle T,\nabla X\rangle,\ldots,\langle T,\nabla X\rangle}_v,\underbrace{\tr(\Theta\nabla^2 X),\ldots,\tr(\Theta\nabla^2 X)}_w\bigr) \\
& = N^{-\frac{1}{2}(k-2)} s^u \Biggl(\prod_{b=u+1}^{u+v} \langle T,\nabla_{t_b}\rangle \prod_{c=u+v+1}^{k} \tr\bigl(\Theta\nabla^2_{t_c}\bigr)\Biggr) \kappa^{(k)}\bigl((\tfrac{1}{2}\Vert t_a-t_b\Vert^2)_{a<b}\bigr)\Big|_{t_1=\cdots=t_{k}}
\end{aligned}
\end{equation}
with $\nabla_{t_b}=(\partial/\partial t_{b}^{i})_{1\le i\le n}$, $\nabla^2_{t_c}=(\partial^2/\partial t_{c}^{i}\partial t_{c}^{j})_{1\le i,j\le n}$. 

For instance, 
\begin{align*}
 N^{-\frac{1}{2}} K_{0,2,1}(s,T,\Theta)
=& \cum(\langle T,\nabla X\rangle,\langle T,\nabla X\rangle,\tr(\Theta \nabla^2 X)) \\
=& \sum_{i,j,k,l=1}^n \tau_i \tau_j \theta_{kl} \cum(X_i,X_j,X_{kl}),
\end{align*}
and, as shown in (\ref{cum_example}),
\begin{align*}
\cum(X_i,X_j,X_{kl})
&= \frac{\partial^4
N^{-\frac{1}{2}} \kappa^{(3)}\bigl(\tfrac12\Vert t_1-t_2\Vert^2,\tfrac12\Vert t_1-t_3\Vert^2,\tfrac12\Vert t_2-t_3\Vert^2\bigr)}{\partial t_1^i\partial t_2^j\partial t_3^k\partial t_3^l} \Big|_{t_1=t_2=t_3} \\
&= N^{-\frac{1}{2}} \bigl(-2\kappa_{11} \delta_{ij}\delta_{kl} + \kappa_{11} (\delta_{ik}\delta_{jl}+\delta_{il}\delta_{jk})\bigr),
\end{align*}
where we let
$\kappa_{11} = \partial^2\kappa^{(3)}(x_{12},x_{13},x_{23})/\partial x_{12}\partial x_{13}|_{x_{12}=x_{13}=x_{23}=0}$.
Thus, we obtain
\[
 K_{0,2,1}(s,T,\Theta)
= -2\kappa_{11}\Vert T\Vert^2\tr(\Theta) + 2 \kappa_{11} T^\top\Theta T.
\]

In this example, the factors $\Vert T\Vert^2\tr(\Theta)$ and $T^\top\Theta T$ are invariant under the transformation $(T,\Theta)\mapsto (P T,P\Theta P^\top)$, $P\in O(n)$.
This results from the isotropic property
$(X,\nabla X,\nabla^2 X)\mathop{=}^d (X,P\nabla X,\allowbreak P\nabla^2 X P^\top)$.
The general form of $K_{u,v,w}(s,T,\Theta)$ is specified by the isotropic assumption.

\begin{lemma}
\label{lem:invariant_poly}
{\rm (i)} $K_{u,2v,w}(s,T,\Theta)$ is a linear combination of
\begin{equation}
\label{K}
 s^u \prod_{j\ge 0}(T^\top \Theta^j T)^{v_j} \prod_{k\ge 1}\tr(\Theta^k)^{w_k},
\end{equation}
where $(v_j)_{j\ge 0}$ and $(w_k)_{k>0}$ are non-negative integers satisfying
\begin{equation*}
 \sum_{j\ge 0}v_j = v, \quad \sum_{j\ge 1} j v_j + \sum_{k\ge 1} k w_k = w.
\end{equation*}
{\rm (ii)} $K_{u,2v+1,w}(s,T,\Theta)=0$.
\end{lemma}

\begin{proof}[Proof of Lemma \ref{lem:invariant_poly}]
Because of the isotropic property, and the multilinearity of the cumulant,
we have
\[
 K_{u,v,w}(s,T,\Theta)=s^u K_{u,v,w}(1,T,\Theta)=s^u K_{u,v,w}(1,P^\top T,P^\top \Theta P), \ \ P\in O(n).
\]
By letting $P=-I_n$, we can observe that $K_{u,v,w}(1,T,\Theta)=0$ when $v$ is odd.
When $v$ is even, $K_{u,v,w}(1,T,\Theta)$ is an even polynomial in $\tau_i$, and is a function of $(\tau_i \tau_j)_{1\le i,j\le n}= T T^\top\in\Sym(n)$ as well.

As an extension of the zonal polynomial in multivariate analysis,
\cite{davis:1980} introduced an invariant polynomial of two symmetric matrices $A,B\in\Sym(n)$ which is invariant under the transformation $(A,B)\mapsto (P^\top A P,\allowbreak P^\top B P)$, $P\in O(n)$.

$K_{u,v,w}(1,T,\Theta)$ is an invariant polynomial in $(T T^\top,\Theta)$, which is a linear combination of (\ref{K}) as shown in (4.8) in \cite{davis:1980}.
\end{proof}

Next, the expression of $K_{u,v,w}$'s in (\ref{Kuvw}) is obtained.
The abbreviations
($\gamma,\ko,\allowbreak \ki,\kii,\kko,\kki,\kkiia,\kkiiaa,\kkiiid,\kkiiia$)
are defined in Table \ref{tab:derivatives}.
Because of Theorem \ref{thm:loop} (to be proved in Section \ref{sec:undetectable}),
the derivatives that have a cycle of the length greater than or equal to $2$
in their diagram do not contribute to the final results,
and are denoted by the symbol ``$\cyc$'' (the explicit form is not required).

$K_{u,2v,w}$ for $k=u+2v+w=2,3,4$ are displayed below.

\subsubsection*{The second order cumulants:}
\begin{align*}
& K_{2,0,0} = s^2, \qquad
K_{0,2,0}
 = \gamma \Vert T\Vert^2, \quad
K_{1,0,1} = -\gamma \tr(\Theta), \\
& K_{0,0,2} = \rho''(0)[2\tr(\Theta^2) + \tr(\Theta)^2] = \cyc.
\end{align*}

\subsubsection*{The third order cumulants:}
\begin{align*}
 K_{3,0,0}
 =& \kappa_0 s^3, \quad
K_{1,2,0}
 = -\kappa_1 s \Vert T\Vert^2, \quad
K_{2,0,1}
= 2\kappa_1 s^2 \tr(\Theta), \\
K_{1,0,2}
 =& 3\kappa_{11} s \tr(\Theta)^2 + \cyc, \quad
K_{0,2,1}
 = -2\kappa_{11} \Vert T\Vert^2 \tr(\Theta)
 + 2\kappa_{11} T^\top\Theta T, \\
K_{0,0,3}
 =& \cyc.
\end{align*}

\subsubsection*{The fourth order cumulants:}
\begin{align*}
%
K_{4,0,0}
 =& \kk_0 s^4, \quad
K_{2,2,0}
 = -\kk_1 s^2 \Vert T\Vert^2, \quad
K_{0,4,0}
 = 3 \kk_{11}^{aa} \Vert T\Vert^4, \\
K_{3,0,1}
 =& 3\kk_1 s^3 \tr(\Theta), \quad
K_{2,0,2}
 = (6\kk_{11}^a+2\kk_{11}^{aa}) s^2 \tr(\Theta)^2 + \cyc, \\
K_{1,2,1}
 =& -(2\kk_{11}^a +\kk_{11}^{aa}) s \Vert T\Vert^2 \tr(\Theta) + 2\kk_{11}^a s T^\top\Theta T, \\
K_{1,0,3}
 =& (12\kk_{111}^a+4\kk_{111}^d) s \tr(\Theta)^3 + \cyc, \\
K_{0,2,2}
 =& -(6\kk_{111}^a +2\kk_{111}^d) \Vert T\Vert^2 \tr(\Theta)^2 + (8\kk_{111}^a +4\kk_{111}^d) T^\top\Theta T \tr(\Theta) \\ & -8\kk_{111}^a T^\top\Theta^2 T + \cyc, \quad
K_{0,0,4}
 = \cyc.
\end{align*}

The term ``$\cyc$'' can be set to be zero in the following calculations.

\subsection{Proof of Theorem \ref{thm:expansion}}
\label{sec:proof_expansion}

In the previous section, the Taylor series of the cumulant generating function $\log\psi_N(s,T,\Theta)$ in (\ref{mgf0}) was obtained up to $O(N^{-1})$.
This is rewritten as the cumulant generating function of $(X,V,R)$, $V=\nabla X$, and $R=\nabla^2 X+\gamma X I_n$.

The distribution of $(X,V,R)$ when $N\to\infty$ is presented in Section \ref{sec:marginal}.
The characteristic functions of $X$, $V$ and $R$ when $N\to\infty$ are
\begin{equation*}
 \psi^0_{X}(s) = e^{-\frac{1}{2}s^2}, \qquad
 \psi^0_{V}(T) = e^{-\frac{1}{2}\gamma \Vert T\Vert^2}, \qquad
 \psi^0_{R}(\Theta) = e^{-\frac{1}{2}\alpha\tr(\Theta^2) -\frac{1}{2}\beta\tr(\Theta)^2},
\end{equation*}
where $\gamma=-\rho'(0)$, $\alpha = 2\rho''(0)$, and $\beta = \rho''(0)-\rho'(0)^2$.

From this definition, the characteristic function of $(X,V,R)$ is
\begin{equation}
\label{mgf}
 \E\Bigl[ e^{\iii(s X+\langle T,V\rangle+\tr(\Theta R))}\Bigr]
 = \psi_N(\s,T,\Theta), \quad \s=\s(s,\Theta)=s+\gamma\tr(\Theta).
\end{equation}
Because we assume that $(X,V,R)$ has fourth moments
(Assumption \ref{assumptions} (ii)), the cumulant generating function is
\begin{align*}
\log\psi_N(\s,T,\Theta)
=& \log\psi^0_{X}(s) + \log\psi^0_{V}(T) +\log\psi^0_{R}(\Theta) \\
 & + \frac{1}{\sqrt{N}} Q_{3}(\s,T,\Theta) + \frac{1}{N} Q_{4}(\s,T,\Theta) + o(N^{-1}),
\end{align*}
where
\begin{align*}
 Q_{k}(\s,T,\Theta) = \sum_{u+v+w=k}\frac{1}{u!\,v!\,w!} K_{u,v,w}(\s,T,\Theta),\ \ k=3,4.
\end{align*}
Therefore,
$\psi_N(\s,T,\Theta) = \widehat\psi_N(\s,T,\Theta) + o(N^{-1})$,
where
\begin{equation}
\label{mgf-hat}
\begin{aligned}
 \widehat\psi_N(\s,T,\Theta)
=& \psi^0_X(s)\psi^0_V(T)\psi^0_R(\Theta) \\
 & \times \Bigl(1 +\frac{1}{\sqrt{N}} Q_{3}(\s,T,\Theta) +\frac{1}{N} Q_{4}(\s,T,\Theta)
 +\frac{1}{2N} Q_{3}(\s,T,\Theta)^2 \Bigr).
\end{aligned}
\end{equation}

Define
\begin{equation}
\label{mgf_V0}
\begin{aligned}
 \psi_{N|V=0}(\s,\Theta)
=& \int_{\R} \int_{\Sym(n)} e^{\iii(s x+\tr(\Theta R))} p_N(x,0,R) \dd R\,\dd x \\
=& \frac{1}{(2\pi)^n}\int_{\R^n} \psi_N(\s,T,\Theta) \dd T.
\end{aligned}
\end{equation}
Its truncated version is
\begin{equation}
\label{mgf_V0-hat}
\begin{aligned}
 \widehat\psi_{N|V=0}(\s,\Theta)
=& \int_{\R} \int_{\Sym(n)} e^{\iii(s x+\tr(\Theta R))} \widehat p_N(x,0,R) \dd R\,\dd x \\
=& \frac{1}{(2\pi)^n}\int_{\R^n} \widehat\psi_N(\s,T,\Theta) \dd T,
\end{aligned}
\end{equation}
where $\widehat p_N(x,V,R)$ denotes the Fourier inversion of $\widehat\psi_N(\s,T,\Theta)$.
The explicit form of $\widehat p_N(x,V,R)$ is not required here.
Recall that the terms in parentheses in (\ref{mgf-hat}) are a linear combination of (\ref{K}).
In the following, we obtain the concrete form of (\ref{mgf_V0-hat}).

In (\ref{mgf_V0-hat}),
the integration with respect to $\dd T$ is conducted as an expectation with respect to the Gaussian random variable $T\sim \mathcal{N}_n(0,\gamma^{-1}I_n)$,
and multiplying it by the normalizing factor
\begin{equation}
\label{normatizing-T}
 \frac{1}{(2\pi)^n}\int_{\R^n}e^{-\frac{1}{2}\gamma\Vert T\Vert^2} \dd T = (2\pi\gamma)^{-n/2}.
\end{equation}
Let
\[
 (n)_m = n(n-1)\cdots(n-m+1) = \Gamma(n+1)/\Gamma(n-m+1)
\]
be the falling factorial.
Because $\Vert T\Vert$ is independent of
$\prod_{j\ge 1} (T^\top\Theta^j T/\Vert T\Vert^2)^{v_j}$
under $T\sim \mathcal{N}_n(0,\gamma^{-1}I_n)$,
the expectation of the factor of (\ref{K}) including $T$ becomes
\begin{align*}
 \E^{T}\Bigl[{\prod}_{j\ge 0}(T^\top \Theta^j T)^{v_j}\Bigr]
=& \E^{T}\bigl[\Vert T\Vert^{2 v}\bigr] \E^{T}\Biggl[\frac{\prod_{j\ge 1}(T^\top \Theta^j T)^{v_j}}{\Vert T\Vert^{2(v-v_0)}}\Biggr] \\
=& \E^{T}\bigl[\Vert T\Vert^{2 v}\bigr] \frac{\E^{T}\bigl[\prod_{j\ge 1}(T^\top \Theta^j T)^{v_j}\bigr]}{\E^{T}\bigl[\Vert T\Vert^{2(v-v_0)}\bigr]} \nonumber \\
=& (-2/\gamma)^{v_0} (-(n/2+v-v_0))_{v_0} \times \zeta_{v_1,v_2,\ldots}(\Theta), \nonumber
\end{align*}
where $v=\sum_{j\ge 0} v_j$ and
\begin{equation*}
 \zeta_{v_1,v_2,\ldots}(\Theta)
 = \E\biggl[{\prod}_{j\ge 1}(\xi^\top \Theta^j \xi)^{v_j}\biggr], \quad \xi\sim \mathcal{N}_n(0,I_n)
\end{equation*}
is a polynomial in $\tr(\Theta^j)$, $j\ge 1$.
For example, $\zeta_{1}(\Theta)=\tr(\Theta)$, $\zeta_{2}(\Theta)=\tr(\Theta^2)$, $\zeta_{1,1}(\Theta)=\tr(\Theta)^2+2\tr(\Theta^2)$. 
Note that $\zeta_{v_1,v_2,\ldots}(\Theta)$ does not include the dimension $n$ explicitly.

The terms in $Q_3$, $Q_4$, and $\frac{1}{2}Q_3^2$ in (\ref{mgf-hat}) including $T$ can be rewritten according to the rules below:
\begin{align}
\begin{aligned}
 \Vert T\Vert^{2} &\mapsto \gamma^{-1} n, \qquad
 \Vert T\Vert^{4} \mapsto \gamma^{-2} n(n+2), \\
 T^\top\Theta^k T &\mapsto \gamma^{-1} \tr(\Theta^k), \quad k=1,2, \\
 \Vert T\Vert^2 \cdot T^\top\Theta T &\mapsto \gamma^{-2}(n+2)\tr(\Theta), \\
 (T^\top\Theta T)^2 &\mapsto \gamma^{-2}[\tr(\Theta)^2+2\tr(\Theta^2)],
\end{aligned}
\label{rule1}
\end{align}
and multiplied by the normalizing factor (\ref{normatizing-T}).
Then, the truncated version of $\psi_{N|V=0}(\s,\Theta)$ in (\ref{mgf_V0}) is obtained as
\begin{equation}
\label{mgf_V0-hat1}
\begin{aligned}
 \widehat\psi_{N|V=0}(\s,\Theta)
 =& \frac{1}{(2\pi)^n}\int_{\R^n} \widehat\psi_N(\s,T,\Theta) \dd T \\
 =& \psi^0_{X}(s) \psi^0_{R}(\Theta) (2\pi\gamma)^{-n/2} \\
  & \times \Bigl(1 + \frac{1}{\sqrt{N}}\widetilde Q_{3}(\s,\Theta) + \frac{1}{N}\widetilde Q_{4}(\s,\Theta) + \frac{1}{2N}\widetilde Q_{3}^{(2)}(\s,\Theta) \Bigr),
\end{aligned}
\end{equation}
where
$\widetilde Q_{3}(\s,\Theta)$, $\widetilde Q_{4}(\s,\Theta)$, and $\widetilde Q_{3}^{(2)}(\s,\Theta)$
are
$Q_{3}(\s,T,\Theta)$, $Q_{4}(\s,T,\Theta)$, and $Q_{3}(\s,T,\Theta)^2$ with terms including $T$ replaced according to the rules (\ref{rule1}).

Recall that the integral we are going to obtain is $\Xi_{n,N}(x)$ in (\ref{Xin}).
Define
\begin{align*}
\psi_{N|x,V=0}(\Theta)
=& \int_{\Sym(n)}e^{\iii\tr(\Theta R)} p_N(x,0,R) \dd R \\
=& \frac{1}{2\pi}\int_{\R} e^{-\iii s x} \psi_{N|V=0}(\s,\Theta) \dd s,
\end{align*}
where $\psi_{N|V=0}(\s,\Theta)$ is defined in (\ref{mgf_V0}).
Using this,
\[
 \int_{\Sym(n)} \det(-R+\gamma x I_n) p_N(x,0,R) \dd R
 = \det\Bigl(-\iiii D_\Theta + \gamma x I_n\Bigr) \Big|_{\Theta=0}\psi_{N|x,V=0}(\Theta),
\]
where $D_\Theta$ is an $n\times n$ symmetric matrix differential operator defined by
\begin{equation}
\label{DTheta}
(D_\Theta)_{ij} = \frac{1+\delta_{ij}}{2}\frac{\partial}{\partial (\Theta)_{ij}} = \frac{\partial}{\partial \tau_{ij}} \quad (i\le j).
\end{equation}
Therefore, (\ref{Xin}) is evaluated as
\begin{equation*}
 \Xi_{n,N}(x) = \int_{x}^\infty \biggl[
 \det\Bigl(-\iiii D_\Theta + \gamma x I_n\Bigr) \Big|_{\Theta=0}
 \frac{1}{2\pi}\int_{\R} e^{-\iii s x} \psi_{N|V=0}(\s,\Theta) \dd s \biggr] \dd x.
\end{equation*}
The truncated version of $\Xi_{n,N}(x)$ is
\begin{equation}
\label{target}
 \widehat\Xi_{n,N}(x) = \int_{x}^\infty \biggl[
 \det\Bigl(-\iiii D_\Theta + \gamma x I_n\Bigr) \Big|_{\Theta=0}
 \frac{1}{2\pi}\int_{\R} e^{-\iii s x} \widehat\psi_{N|V=0}(\s,\Theta) \dd s \biggr] \dd x,
\end{equation}
which is the valid asymptotic expansion formula for $\Xi_{n,N}(x)$ as follows.
\begin{lemma}
\label{lem:validity}
Under Assumption \ref{assumptions},
\[
 \Xi_{n,N}(x)=\widehat\Xi_{n,N}(x)+o(N^{-1}) \quad \mbox{as $N\to\infty$ uniformly in $x$}.
\]
\end{lemma}
The proof is provided in Section \ref{sec:asympt}.

Lemma \ref{lem:validity} states that our target is $\widehat\Xi_{n,N}(x)$.
The integral in (\ref{target}) with respect to $\dd s$ can be easily evaluated by
\begin{equation}
\label{s2x}
 \frac{1}{2\pi}\int_{\R} e^{-\iii s x} \psi^0_{X}(s) (\iii s)^k \dd s = H_k(x) \phi(x),
\end{equation}
where $\phi(x)$ is the probability density function of the standard Gaussian distribution $\mathcal{N}(0,1)$ in (\ref{gauss_pdf}), and
$H_k(x)$ is the Hermite polynomial of degree $k$ defined in (\ref{hermite}).
For the derivatives with respect to $\Theta$, we use the lemma below.
The proof is presented in Section \ref{sec:Hermite}.

\begin{lemma}
\label{lem:DThetadet}
For $\psi^0_{R}(\Theta) = e^{-\frac{1}{2}\alpha\tr(\Theta^2) -\frac{1}{2}\beta\tr(\Theta)^2}$, $\gamma=\sqrt{\alpha/2-\beta}$,
and positive integers $c_i$ such that $m=\sum_{i=1}^k c_i\le n$,
\begin{equation}
\label{theta2x}
\begin{aligned}
 \det\Bigl(-\iiii D_\Theta + \gamma x I_n\Bigr) \bigl(\psi^0_R(\Theta)\tr(\Theta^{c_1})\cdots\tr(\Theta^{c_k})\bigr)\Big|_{\Theta=0} \\ 
= \iii^m \gamma^{n-m} (-1/2)^{m-k} (n)_m H_{n-m}(x).
\end{aligned}
\end{equation}
\end{lemma}

Now, we summarize the entire procedure for obtaining $\widehat\Xi_{n,N}(x)$ in (\ref{target}).
\begin{itemize}

\item[Step 0.]
Express $K_{u,v,w}(s,T,\Theta)$ with the base functions in (\ref{K}) and the derivatives of $\kappa^{(k)}$, $k=2,3,4$
 (executed in Section \ref{sec:iso_cumulants}).

\item[Step 1.]
Expand the inside of the parentheses of $\widehat\psi_N(\s,T,\Theta)$ in (\ref{mgf-hat}) and applying the term-rewriting rules in (\ref{rule1}) to obtain $\widehat\psi_{N|V=0}(\s,\Theta)$ in (\ref{mgf_V0-hat1}).
The resulting function is the product of $\psi^0_X(s)\psi^0_R(\Theta)$ and a polynomial in $s$ and $\tr(\Theta^k)$, $k\le 4$.

\item[Step 2.]
Applying (\ref{s2x}) and (\ref{theta2x}) to the result of Step 1 yields the inside brackets in (\ref{target}),
which is $\phi(x)$ multiplied by a polynomial in $x^k$ and $H_{k}(x)$.

\item[Step 3.]
By using the three-term relation
$x H_{k}(x) = H_{k+1}(x) + k H_{k-1}(x)$ in (\ref{three-term}),
we reform the result of Step 2 to be $\phi(x)$ multiplied by a linear combination of $H_{k}(x)$.

\item[Step 4.]
Applying the integration $\int_{x}^\infty \dd x$ to the result of Step 3, by using
\[
 \int_{x}^\infty H_k (x') \phi(x') \dd x' = H_{k-1}(x) \phi(x),
\]
yields the integral $\widehat\Xi_{n,N}(x)$ in (\ref{target}).
Let $H_{-1}(x) = \phi(x)^{-1}\int_x^\infty \phi(x') \dd x'$.
\end{itemize}

Every step is doable using computational algebra.
By performing all the steps, the proof of Theorem \ref{thm:expansion} is completed.

\subsection{Undetectable non-Gaussianity and proof of Theorem \ref{thm:loop}}
\label{sec:undetectable}

In this section, we prove Theorem \ref{thm:loop}.

For $\ell\ge 2$, let
\[
 \Pi_\ell(\Theta) = \tr(\Theta)^\ell - (-2)^{\ell-1}\tr(\Theta^\ell).
\]
Because of Lemma \ref{lem:DThetadet},
for positive integers $c_i$ such that $m=\sum_{i=1}^k c_i\le n-\ell$,
\begin{equation*}
\begin{aligned}
\det & \Bigl(-\iiii D_\Theta + \gamma x I_n\Bigr) \bigl(\psi^0_R(\Theta)\tr(\Theta^{c_1})\cdots\tr(\Theta^{c_k}) \Pi_\ell(\Theta) \bigr)\Big|_{\Theta=0} \\
=& \iii^{m+\ell} \gamma^{n-(m+\ell)} (-1/2)^{(m+\ell)-(k+\ell)} (n)_{m+\ell} H_{n-(m+\ell)}(x) \\
& -(-2)^{\ell-1} \iii^{m+\ell} \gamma^{n-(m+\ell)} (-1/2)^{(m+\ell)-(k+1)} (n)_{m+\ell} H_{n-(m+\ell)}(x) \\
=& 0.
\end{aligned}
\end{equation*}
Therefore, if terms containing the factor $\Pi_\ell(\Theta)$ exist, they automatically vanish in Step 2.
Actually, such a term appears.
For instance,
\begin{equation}
\label{K003} 
\begin{aligned}
K_{0,0,3}(s,T,\Theta)
&= \tr(\Theta\nabla^2_{t_1})\tr(\Theta\nabla^2_{t_2})\tr(\Theta\nabla^2_{t_3})
 \kappa^{(3)}\bigl((\tfrac{1}{2}\Vert t_a-t_b\Vert^2)_{1\le a<b\le 3}\bigr)\big|_{t_1=t_2=t_3} \\
&= 6\kappa^{(3)}_{(12),(12),(23)}(0) \Pi_2(\Theta)\tr(\Theta) + 2\kappa^{(3)}_{(12),(13),(23)}(0) \Pi_3(\Theta).
\end{aligned}
\end{equation} 
Each term on the right-hand side in (\ref{K003}) includes a factor $\Pi_\ell(\Theta)$ and a derivatives of $\kappa^{(3)}$ which has a cycle in its diagram (Figure \ref{fig:loop+tree} (a) and (b)).

\begin{proof}[Proof of Theorem \ref{thm:loop}]

Let $x_{ab}=\frac{1}{2}\Vert t_a-t_b\Vert^2$.
Suppose that the diagram of the derivative
\[
 \kappa^{(k)}_E(0) = \Biggl(\prod_{(a,b)\in E}\Bigl(\frac{\partial}{\partial x_{ab}}\Bigr)\Biggr) \kappa^{(k)}\bigl((x_{ab})_{1\le a<b\le k}\bigr)\Big|_{t_1=\cdots=t_k}
\]
contains a cycle
$C = \{(1,2),\ldots,(\ell-1,\ell),(1,\ell)\} \subset E$.

The derivative $\kappa^{(k)}_E(0)$ appears in the Taylor series
\begin{align}
\label{taylor}
 \kappa^{(k)}\bigl((x_{ab})_{1\le a<b}\bigr)
 = \cdots + \kappa^{(k)}_E(0)\times \Delta\widetilde\Delta + \cdots,
\end{align}
where $\Delta=x_{12}\cdots x_{\ell-1,\ell}x_{1,\ell}$ and $\widetilde\Delta=\prod_{(a,b)\in E\setminus C}x_{ab}$.

We consider the application of operators $\langle T,\nabla_{t_b}\rangle$ and/or $\tr(\Theta\nabla^2_{t_c})$ to (\ref{taylor}), and evaluate it at $t_1=\cdots=t_k$. 
After applying these differentiation operations, the coefficient $\Delta\widetilde\Delta$ of $\kappa^{(k)}_E(0)$ should be reduced to a nonzero constant.
(Otherwise, it vanishes when evaluated at $t_1=\cdots=t_k$.)
At least, $\Delta$ should be reduced to a nonzero constant.

The only possible operation that makes $\Delta$ a nonzero constant is $\prod_{a=1}^\ell \tr(\Theta\nabla^2_{t_a})$:
\begin{equation}
\begin{aligned}
\label{trtr}
& \tr(\Theta\nabla^2_{t_1}) \cdots \tr(\Theta\nabla^2_{t_\ell})
 (x_{12}\cdots x_{\ell-1,\ell}x_{1,\ell}) \\
&= \sum_{{i_1,\ldots,i_v,j_1,\ldots,j_v}=1}^n
 \theta_{i_1 j_1}\cdots\theta_{i_v j_v}
 \prod_{a=1}^\ell \frac{\partial^{2}}{\partial t_{a}^{i_a}\partial t_{a}^{j_a}}
 (x_{12}\cdots x_{\ell-1,\ell}x_{1,\ell}).
\end{aligned}
\end{equation}
By selecting the non-vanishing terms, we obtain
\begin{align*}
\prod_{a=1}^\ell & \frac{\partial^{2}}{\partial t_{a}^{i_a}\partial t_{a}^{j_a}}
 (x_{12}\cdots x_{\ell-1,\ell}x_{1,\ell}) 
= \prod_{a=1}^\ell \frac{\partial^{2}x_{a,a+1}}{\partial t_{a}^{i_a}\partial t_{a}^{j_a}}
 +\prod_{a=1}^\ell \frac{\partial^{2}x_{a,a+1}}{\partial t_{a+1}^{i_{a+1}}\partial t_{a+1}^{j_{a+1}}} \\
&\qquad + \sum_{(\epsilon_1,\ldots,\epsilon_\ell)\in\{0,1\}^\ell} \prod_{a=1}^\ell\biggl((1-\epsilon_a)\frac{\partial^{2}x_{a,a+1}}{\partial t_{a}^{i_a}\partial t_{a+1}^{j_{a+1}}}
 +\epsilon_a \frac{\partial^{2}x_{a,a+1}}{\partial t_{a+1}^{i_{a+1}}\partial t_{a}^{j_a}}\biggr) \\
&= 2\prod_{a=1}^\ell \delta_{i_a j_a} + (-1)^\ell \prod_{a=1}^\ell \bigl(\delta_{i_{a} j_{a+1}} + \delta_{i_{a+1} j_{a}} \bigr)
\end{align*}
(letting $x_{\ell,\ell+1}=x_{1,\ell}$ and $i_{\ell+1}=i_1$, $j_{\ell+1}=j_1$),
and hence
\[
 (\ref{trtr}) = 2\tr(\Theta)^\ell - (-1)^{\ell}2^\ell \tr(\Theta^\ell) = 2\Pi_\ell(\Theta).
\]
This implies that the coefficients of $\kappa^{(k)}_E(0)$ have the factor $\Pi_\ell(\Theta)$.
\end{proof}

\appendix

\section{}

\subsection{Identities on the Hermite polynomial and proof of
 Lemma \ref{lem:DThetadet}}
\label{sec:Hermite}

We propose the identities of the Hermite polynomial, which is crucial in the derivation of the expansion.

For an $n\times n$ symmetric matrix $A$, the principal minor matrix corresponding to indices $K\subset\{1,\ldots,n\}$ is denoted by
$A[K] = (a_{ij})_{i,j\in K}$.
We note that $A=A[\{1,\ldots,n\}]$.

We first prove the following lemma.
Recall that $\Theta=(\theta_{ij})$ and $D_\Theta=(d_{ij})$ are defined in (\ref{Theta}) and (\ref{DTheta}), respectively.

\begin{lemma}
\label{lem:DThetadet1}
For positive integers $c_i$ such that $m=\sum_{i=1}^k c_i\le n$,
\begin{equation}
 \det(x I + D_\Theta) \bigl(e^{\tr(\Theta^2)}\tr(\Theta^{c_1})\cdots\tr(\Theta^{c_\ell})\bigr)\Big|_{\Theta=0} = (-1/2)^{m-\ell} (n)_m H_{n-m}(x).
\label{theta2x1}
\end{equation}
\end{lemma}

\begin{proof}
Based on the expansion formula
\[
 \det(x I_n + D_\Theta) = \sum_{k=0}^n x^{n-k} \sum_{K:K\subset\{1,\ldots,n\},\,|K|=k} \det\bigl(D_{\Theta[K]}\bigr),
\]
the left-hand side of (\ref{theta2x1}) is a polynomial in $x$ with coefficients in the form of
\begin{equation}
\label{detDK-e-tr}
 \det(D_{\Theta[K]}) \bigl(e^{\tr(\Theta^2)}\tr(\Theta^{c_1})\cdots\tr(\Theta^{c_\ell})\bigr)\Big|_{\Theta=0}.
\end{equation}
By symmetry, it suffices to consider the case $K=\{1,\ldots,k\}$.
Let $\Theta_k=\Theta[\{1,\ldots,k\}]$.

By the definition of the determinant,
\begin{equation*}
 \det(D_{\Theta_k}) = \sum_{\sigma\in S_k} \sgn(\sigma) d_{1\sigma(1)}\cdots d_{n\sigma(k)},
\end{equation*}
where $S_k$ denotes the permutation group on $\{1,\ldots,k\}$.

$\tr(\Theta^c)$ is a linear combination of the terms of the form
$\theta_{j_1 j_2}\theta_{j_2 j_3}\cdots\theta_{j_{c-1} j_c}\theta_{j_c j_1}$.
The form
\begin{equation}
\label{d-tr}
 d_{i_1\sigma(i_1)}\cdots d_{i_{e}\sigma(i_{e})} \bigl(\theta_{j_1 j_2}\theta_{j_2 j_3}\cdots\theta_{j_{c-1} j_c}\theta_{j_c j_1}\bigr) \Big|_{\Theta=0}
\quad (1\le i_1<\cdots<i_{e}\le k)
\end{equation}
does not vanish iff $e=c$, the map $\begin{pmatrix} i_1 & \cdots & i_{c} \\ \sigma(i_1) & \cdots & \sigma(i_{c}) \end{pmatrix}$ forms a cycle of length $c$,
and
\[
 (j_1,\ldots,j_c) = (i_h,\sigma(i_h),\sigma^2(i_h)\ldots,\sigma^{c-1}(i_h)) \ \ \mbox{or}\ \ (i_h,\sigma^{-1}(i_h),\ldots,\sigma^{-(c-1)}(i_h))
\]
for some  $h=1,\ldots,c$ (i.e., there are $2c$ ways).
The value of (\ref{d-tr}) is $(1/2)^c$ if it does not vanish.

The form
\begin{equation}
\label{d-e}
 d_{i_1\sigma(i_1)}\cdots d_{i_{e}\sigma(i_{e})}e^{\tr(\Theta^2)} \Big|_{\Theta=0} \quad (1\le i_1<\cdots<i_{e}\le k)
\end{equation}
does not vanish iff $e$ is even, and the map $\begin{pmatrix} i_1 & \cdots & i_{e} \\ \sigma(i_1) & \cdots & \sigma(i_{e}) \end{pmatrix}$ is a product of $e/2$ cycles of length 2.
The value of (\ref{d-e}) is $1$ if it does not vanish.

Therefore,
\[
 \sgn(\sigma) d_{1\sigma(1)}\cdots d_{k\sigma(k)} \bigl(e^{\tr(\Theta^2)}\tr(\Theta^{c_1})\cdots\tr(\Theta^{c_\ell})\bigr)\Big|_{\Theta=0}
\]
does not vanish iff $\sigma$ (in the cycle product form) is factorized as
$\ell$ cycles of length $c_i$, $i=1,\ldots,\ell$, and $e/2$ cycles of length 2, where $e=k-\sum_{i=1}^\ell c_i=k-m$ is even.
Note that the number of cycles made from distinct $c$ atoms is $(c-1)!$,
and the number of such $\sigma$ is
\begin{align*}
& \binom{k-c_1}{c_1} (c_1-1)! \times \binom{k-c_1-c_2}{c_2} (c_2-1)! \times\cdots \\ & \qquad \times 
 \binom{k-c_1-\cdots -c_{\ell-1}}{c_\ell} (c_\ell-1)! \times (k-m)!! \\
& = \frac{(k)_m}{\prod c_i} \times\frac{(k-m)!}{2^{\frac{1}{2}(k-m)}(\frac{k-m}{2})!}
  = \frac{k!}{(\prod c_i) 2^{\frac{1}{2}(k-m)}(\frac{k-m}{2})!}.
\end{align*}
The sign of $\sigma$ is
\[
 \sgn(\sigma)=\prod_{i=1}^\ell (-1)^{c_i-1} \times (-1)^{(k-m)/2}
 = (-1)^{m-\ell+(k-m)/2}.
\]

Therefore,
\begin{equation*}
 (\mathrm{\ref{detDK-e-tr}})
 = \frac{k!}{(\prod c_i) 2^{\frac{1}{2}(k-m)}(\frac{k-m}{2})!}
 \times \prod_{i=1}^\ell \bigl( 2 c_i\times (1/2)^{c_i}\bigr) \times 1 
 = \frac{k!}{2^{m-\ell+(k-m)/2}(\frac{k-m}{2})!}.
\end{equation*}

We now show that the left-hand side of (\ref{theta2x1}) is
\begin{align*}
& \sum_{k=m,\,k-m:\rm even}^n x^{n-k} \binom{n}{k} \frac{k!}{2^{m-\ell+(k-m)/2}(\frac{k-m}{2})!} (-1)^{m-\ell+(k-m)/2} \\
& = (-1/2)^{m-\ell}\sum_{k'=0}^{[\frac{n-m}{2}]} x^{n-m-2k'} \frac{n!}{2^{k'}(n-m-2k')!k'!} (-1)^{k'} \qquad \Bigl(k'=\frac{k-m}{2}\Bigr) \\
& = (-1/2)^{m-\ell} (n)_m H_{n-m}(x).
\end{align*}
\end{proof}

\begin{lemma}
\label{lem:DThetadet2}
For any $\beta$,
and positive integers $c_i$ such that $m=\sum_{i=1}^k c_i\le n$,
\begin{equation}
\label{theta2x2}
\begin{aligned}
 \det(x I + D_\Theta) \bigl(e^{(1+\beta)\tr(\Theta^2)+\frac{\beta}{2}\tr(\Theta)^2}\tr(\Theta^{c_1})\cdots\tr(\Theta^{c_\ell})\bigr)\Big|_{\Theta=0} \\
 = (-1/2)^{m-\ell} (n)_m H_{n-m}(x).
\end{aligned}
\end{equation}
\end{lemma}

\begin{proof}
(\ref{theta2x2}) with $\beta=0$ holds from Lemma \ref{lem:DThetadet1}.
We demonstrate that the left-hand side of (\ref{theta2x2}) is independent of $\beta$.
Using the expansion around $\beta=0$,
\begin{align*}
 e^{\beta (\tr(\Theta^2)+\frac{1}{2}\tr(\Theta)^2)}
 = \sum_{k\ge 0}\frac{\beta^k}{k!} \sum_{0\le h\le k}\binom{k}{h}\tr(\Theta^2)^{k-h}\tr(\Theta)^{2h} (1/2)^h,
\end{align*}
the left-hand side of (\ref{theta2x2}) becomes a series in $\beta$.
The coefficient of $\beta^k/k!$ is
\begin{align*}
 \sum_{0\le h\le k} \binom{k}{h} (-1/2)^{2k-(k+h)} (n)_{2k} H_{n-2k}(x) (1/2)^h =0
\end{align*}
except when $k=0$.
\end{proof}

\begin{proof}[Proof of Lemma \ref{lem:DThetadet}]
Recall that $\gamma=\sqrt{\alpha/2-\beta}$.
(\ref{theta2x2}) in Lemma \ref{lem:DThetadet2} with $\beta:=\beta/\gamma^2$, $\Theta:=\iii\gamma\Theta$, and $x:=-x$ yields (\ref{theta2x}).
\end{proof}

\subsection{Conditional asymptotic expansion and proof of Lemma \ref{lem:validity}}
\label{sec:asympt}

We begin by summarizing the asymptotic expansion for the probability density function and the moment in the i.i.d.\ setting of
\cite{bhattacharya-Rao:2010}. 

Let $q_1(x)$, $x\in\R^k$, be the probability density of a random vector $X$ with a zero mean and covariance $\Sigma\succ 0$.
Let $\psi_1(t)$ be the characteristic function of $X$. 
Let $q_N(x)$, $x\in\R^k$, be the probability density with the characteristic function $\psi_N(t)=\psi_1(t/\sqrt{N})^{N}$.
Let $\phi_k(x;\Sigma)$ be the probability density of the Gaussian distribution $\mathcal{N}_k(0,\Sigma)$.
Assume that the $s$th moment exists under $q_1$.
Then,
\[
 \log\psi_N(t) = N \log\psi_1\Bigl(\frac{t}{\sqrt{N}}\Bigr)
 = -\frac{1}{2}t^\top\Sigma t + \sum_{j=3}^s \frac{N^{-\frac{1}{2}(j-2)}\iii^j}{j!} \sum_{i:|i|=j} c_i t^i + N r_s\Bigl(\frac{t}{\sqrt{N}}\Bigr),
\]
where $t=(t^1,\ldots,t^k)^\top$ and $i=(i_1,\ldots,i_k)$ is a multi-index such that
$c_i=(c_{i_1},\ldots,c_{i_k})\in\R^k$,
$t^i=(t^1)^{i_1}\cdots (t^k)^{i_k}$,
and $r_s(t)$ is a function such that $r_s(t)=o(|t|^{s-2})$.
Let
\begin{equation}
\label{truncated_mgf}
 \widehat\psi_N^{(s)}(t) = e^{-\frac{1}{2}t^\top\Sigma t}\Biggl(1 + \sum_{j=3}^s N^{-\frac{1}{2}(j-2)} \iii^j F_{3(j-2)}(t)\Biggr),
\end{equation}
where the inside of the parentheses represents the expansion of
\[
 \exp\Biggr(\sum_{j=3}^s \frac{N^{-\frac{1}{2}(j-2)}\iii^j}{j!} \sum_{i:|i|=j} c_i t^i\Biggl)
\]
around $N=\infty$ up to the order of $N^{-\frac{1}{2}(s-2)}$.
Here, $F_{3(j-2)}(t)$ is an even or odd polynomial in $t$ of degree $3(j-2)$.
Let
\begin{equation}
\label{truncated_pdf}
 \widehat q_N^{(s)}(x) = \phi_k(x;\Sigma)\biggl(1 + \sum_{j=3}^s N^{-\frac{1}{2}(j-2)} G_{3(j-2)}(x)\biggr), \quad x=(x_1,\ldots,x_k),
\end{equation}
be the Fourier inversion of $\widehat\psi_N^{(s)}(t)$.
$G_{3(j-2)}(x)$ is an even or odd polynomial in $x$ of degree $3(j-2)$.

\begin{proposition}[{Corollary to \cite[Theorems 19.1 and 19.2]{bhattacharya-Rao:2010}}]
\label{prop:B-R}
Assume that the probability density function $q_N(x)$ exists for $N\ge 1$,
and is bounded for some $N$.
Assume that $\E[\Vert X\Vert^s]<\infty$ under $q_1$.
Then,
$q_N(x)$ is continuous for a sufficiently large $N$, and
\[
 \sup_{x\in\R^k}(1+\Vert x\Vert^s)\bigl|q_N(x)-\widehat q_N^{(s)}(x)\bigr| = o\bigl(N^{-\frac{1}{2}(s-2)}\bigr) \quad\mbox{as}\ N\to\infty.
\]
\end{proposition}

Write $x=(x_1,x_2)$, $x_1\in\R^{k_1}$, $x_2\in\R^{k_2}$ ($k_1+k_2=k$).
In the following, $x_2$ is assumed to be a constant vector $x_{20}$.

\begin{corollary}
\label{cor:B-R}
Let $f(x_1)$ be a function such that $|f(x_1)|\le C (1+\Vert x_1\Vert^{s_1})$.
For $s\ge s_1+k_1+1$ and for $s_0\le s$,
\[
 \int_{\R^{k_1}}f(x_1) q_N(x_1,x_{20}) \dd x_1 = \int_{\R^{k_1}}f(x_1) \widehat q_N^{(s_0)}(x_1,x_{20}) \dd x_1 +o\bigl(N^{-\frac{1}{2}(s_0-2)}\bigr).
\]
\end{corollary}

\begin{proof}
\begin{align*}
 \biggl|\int_{\R^{k_1}} f(x_1) q_N(x_1,x_{20}) \dd x_1 - & \int_{\R^{k_1}}f(x_1) \widehat q_N^{(s_0)}(x_1,x_{20}) \dd x_1 \biggr| \\
\le&  \int_{\R^{k_1}}|f(x_1)| |q_N(x_1,x_{20}) - \widehat q_N^{(s)}(x_1,x_{20}) | \dd x_1
\\
 &+ \int_{\R^{k_1}}|f(x_1)| |\widehat q_N^{(s)}(x_1,x_{20}) - \widehat q_N^{(s_0)}(x_1,x_{20})| \dd x_1.
\end{align*}
The first term is bounded above by
\[
 o\bigl(N^{-\frac{1}{2}(s-2)}\bigr)\times
 C \int_{\R^{k_1}} \frac{1+\Vert (x_1,x_{20})\Vert^{s_1}}{1+\Vert (x_1,x_{20})\Vert^{s}} \dd x_1.
\]
This integral exists when $s-s_1-(k_1-1)>1$.

For the second term, because
$\widehat q_N^{(s)}(x_1,x_{20}) - \widehat q_N^{(s_0)}(x_1,x_{20})$ is
\[
 \phi_k((x_1,x_{20});\Sigma)\times (\mbox{a polynomial in $x_1$}),
\]
the integral exists, and the coefficients of the polynomials are multiples of \\
$N^{-\frac{1}{2}(s_0+1-2)},\ldots,N^{-\frac{1}{2}(s-2)}$ (if $s_0<s$), or 0 (if $s_0=s$). 
Therefore, the second term is $O\bigl(N^{-\frac{1}{2}(s_0-1)}\bigr)$ when $s_0<s$.

The sum of the first and second term is
$o\bigl(N^{-\frac{1}{2}(s-2)}\bigr) + O\bigl(N^{-\frac{1}{2}(s_0-1)}\bigr)\1_{\{s_0<s\}}=o\bigl(N^{-\frac{1}{2}(s_0-2)}\bigr)$.
\end{proof}

\begin{proof}[Proof of Lemma \ref{lem:validity}]
We apply Corollary \ref{cor:B-R} to evaluate (\ref{Xin}).
The characteristic function $\psi_N(t)$ and its truncated version
$\widehat\psi_N^{(s)}(t)$ in (\ref{truncated_mgf}) are given by
 $\psi_N(\s,T,\Theta)$ in (\ref{mgf}) and $\widehat\psi_N(\s,T,\Theta)$ in (\ref{mgf-hat}), respectively.
$q_N(x)$ and its truncated version $\widehat q_N^{(s)}(x)$ in (\ref{truncated_pdf}) are $p_N(x,V,R)$ and $\widehat p_N(x,V,R)$ used in (\ref{mgf_V0-hat}), respectively.
Let $x_1=(X,R)$, $x_2=V=x_{20}=0$, and 
$f(x_1)=\1_{\{X\ge x\}}\det(-R+\gamma X I)$.
Then,
$\int_{\R^{k_1}} f(x_1) q_N(x_1,0) \dd x_1 = \Xi_{n,N}(x)$ and
$\int_{\R^{k_1}} f(x_1) \widehat q_N(x_1,0) \dd x_1 = \widehat \Xi_{n,N}(x)$.
Here, $k_1=1+n(n+1)/2$, $k_2=n$, and $s_1=n$.
Note that $s_1+k_1+1=\binom{n+2}{2}+1$.
The constant $C$ in $|f(x_1)|\le C(1+\Vert x_1\Vert^{s_1})$ in Corollary \ref{cor:B-R} can be chosen independently of $x$.
Hence, the remainder term is irrespective of $x$.

In Theorem \ref{thm:expansion}, we choose $s_0=4$.
If the joint density is bounded and has a moment of the order
$s\ge \max(\binom{n+2}{2}+1,s_0)=\binom{n+2}{2}+1$, the remaining term is $o\bigl(N^{-\frac{1}{2}(s_0-2)}\bigr)=o(N^{-1})$ at least, and the manipulation of asymptotic expansion is validated.
\end{proof}

\bibliographystyle{amsalpha}
\bibliography{minkowski-bib-v3.bib}
\end{document}